\documentclass[11pt,letterpaper]{article}
\usepackage[english]{babel}
\usepackage{amsmath}
\usepackage{amsfonts, mathrsfs}
\usepackage{amssymb}
\usepackage{verbatim}
\usepackage{graphicx}
\usepackage{color}
\usepackage[numbers]{natbib}
\usepackage{url}
\usepackage{subcaption}

\numberwithin{equation}{section}
\newtheorem{theo}{Theorem}[section]

\newtheorem{lemma}[theo]{Lemma}
\newtheorem{example}[theo]{Example}

\newtheorem{assumption}[theo]{Assumption}
\newtheorem{defn}[theo]{Definition}
\newtheorem{remark}[theo]{Remark}

\newenvironment{proof}[1][Proof]{\textbf{#1.} }{\ \rule{0.5em}{0.5em}}

\newcommand{\cov}{{\rm Cov} \mspace{1mu}}

\newcommand{\norm}[1]{\left\lVert#1\right\rVert}

\allowdisplaybreaks

\author{Mariana Olvera-Cravioto\\ \\ University of North Carolina at Chapel Hill}

\title{PageRank's behavior under degree-degree correlations }

\date{}

\begin{document}
\maketitle

\begin{abstract}
The focus of this work is the asymptotic analysis of the tail distribution of Google's PageRank algorithm on large scale-free directed networks. In particular, the main theorem provides the convergence, in the Kantorovich-Rubinstein metric, of the rank of a randomly chosen vertex in graphs generated via either a directed configuration model or an inhomogeneous random digraph. The theorem fully characterizes the limiting distribution by expressing it as a random sum of i.i.d.~copies of the  attracting endogenous solution to a branching distributional fixed-point equation.   In addition, we provide the asymptotic tail behavior of the limit and use it to explain the effect that in-degree/out-degree correlations in the underlying graph can have on the qualitative performance of PageRank. 
\vspace{5mm}



\end{abstract}

\section{Introduction}

Google's PageRank algorithm \cite{Bri_Pag_98}, originally created to rank webpages in the World Wide Web, is arguably one of the most widely used measures of network centrality. At its core, it is the solution to a large system of linear equations which assigns to each webpage (vertex on a directed graph) a universal {\em rank} or {\em score}, which can then be used to determine the order in which results to a specific query will be displayed. PageRank's popularity is due in part to the fact that it can be efficiently computed even on very large networks, but perhaps more importantly, on its ability to identify ``influential" vertices. The aim of this work, as well as of much of the earlier work on the distribution of the ranks produced by PageRank \cite{Lit_Sch_Vol_07,Vol_Lit_Don_07,Volk_Litv_10,Jel_Olv_10,Chen_Lit_Olv_17,Lee_Olv_19}, is to provide some mathematical interpretation to what PageRank is actually ranking highly, and how this is related to the underlying graph where it is being computed. 

One of the early observations made relating the scores produced by PageRank and the graph where they were computed, was that on scale-free graphs (i.e., those whose degree distributions follow a power-law) the distribution of the ranks and that of the in-degree, seem to be proportional to each other.  This observation led to the so-called {\em Power-law Hypothesis}, which states that on a graph whose in-degree distribution follows a power-law, the PageRank distribution will also follow a power-law with the same tail index. This fact was first proved for trees in \cite{Volk_Litv_10,Jel_Olv_10}, then for graphs generated via the directed configuration model (DCM) in \cite{Chen_Lit_Olv_17}, and more recently, for graphs generated via the inhomogeneous random digraph model (IRD) in \cite{Lee_Olv_19}. In addition, the very recent work in \cite{Gar_vdH_Lit_19} shows that if the underlying graph converges in the local weak sense, then the PageRank distribution on the graph converges to the distribution of PageRank computed on its limit (usually a tree). From that result one can obtain a power-law lower bound for the PageRank tail distribution for a wide class of scale-free graphs. The idea behind the results stated above is that the PageRank of a vertex is mostly determined by its immediate inbound neighborhood, so as long as this neighborhood looks locally like a tree, the PageRank distribution on the graph and on the limiting tree will be essentially the same. Once the analysis on a graph is reduced to analyzing a tree, then we can study the asymptotic behavior of the PageRank distribution using the theory of weighted branching processes \cite{Durr_Ligg_83,Holl_Ligg_81,rosler1,Way_Will_95} and distributional fixed-point equations \cite{Jel_Olv_10,Jel_Olv_12a,Jel_Olv_12b,Als_Dam_Men_12,Alsm_Mein_12,Alsm_Mein_13,Alsm_Dysz_15}. 

One of the main contributions of the work in \cite{Volk_Litv_10,Jel_Olv_10,Chen_Lit_Olv_17, Lee_Olv_19}, was the characterization of the limiting PageRank distribution as the solution to a branching distributional fixed-point equation of the form:
\begin{equation} \label{eq:LinearSFPEInd}
\mathcal{R} \stackrel{\mathcal{D}}{=} \sum_{i=1}^{\mathcal{N}} \mathcal{C}_i \mathcal{R}_i + \mathcal{Q},
\end{equation}
where $\mathcal{R}$ represents the (limiting) rank of a randomly chosen vertex, $\mathcal{N}$ represents its in-degree, $\mathcal{Q}$ its personalization value, and the $\{ \mathcal{R}_i\}$ the ranks of its inbound neighbors; the weights $\{\mathcal{C}_i\}$ are related to the out-degree of the graph and the damping factor (see Section~\ref{S.PageRank} for more details) and $\stackrel{\mathcal{D}}{=}$ denotes equality in distribution. Moreover, this characterization can be used to obtain a {\em large deviations} explanation of what PageRank is ranking highly (we write $f(x)\sim g(x)$ as $x \to \infty$ to mean $\lim_{x \to \infty} f(x)/g(x) = 1$). Specifically, the work in \cite{Jel_Olv_10, Olvera_12b} shows that whenever $\mathcal{N}$ follows a power-law, then
\begin{equation} \label{eq:TailLDind}
P(\mathcal{R} > x) \sim P\left( \mathcal{N} > x/E[\mathcal{C}_1 \mathcal{R}_1] \right) + P\left( \max_{1 \leq i \leq \mathcal{N}} \mathcal{C}_i \mathcal{R}_i > x \right) , \qquad x \to \infty.
\end{equation}
In other words, the most likely way in which a vertex can achieve a high rank is either by having a very large in-degree, or by having a highly ranked inbound neighbor. The proof of the Power-law Hypothesis can be derived by iterating a version of \eqref{eq:TailLDind}, as done in \cite{Jel_Olv_10, Olvera_12b}, or through the use of transforms as in \cite{Volk_Litv_10}. Either way, one shows that the probabilities on the right-hand side of \eqref{eq:TailLDind} are proportional to each other. 

However, one limitation of the results in \cite{Chen_Lit_Olv_17, Lee_Olv_19} is that they only cover graphs where the in-degree and out-degree of each vertex are asymptotically independent, which is not necessarily the case in real-world networks. This paper is aimed at completing the analysis of the PageRank distribution under the most general assumptions possible for the in-degree, out-degree, personalization value and damping factor, while still preserving a full characterization of the limit as well as of its asymptotic tail behavior. In view of this goal, we focus only on two random graph models which converge, in the local weak sense, to a marked Galton-Watson process, since as explained in Section~\ref{SS.WeightedTree}, it is the natural structure where the solutions to branching distributional fixed-point equations can be constructed. 

The main contributions of this paper are two-fold. First, we provide a full generalization of the main theorems in \cite{Chen_Lit_Olv_17, Lee_Olv_19} that holds for arbitrarily dependent $(\mathcal{Q}, \mathcal{N}, \{ \mathcal{C}_i\})$, and that shows that the PageRank of a randomly chosen vertex in a graph generated via either the DCM or IRD models, converges, in the Kantorovich-Rubinstein metric, to a random variable $\mathcal{R}^*$. This random variable $\mathcal{R}^*$ can be written as a random sum of i.i.d.~copies of the attracting endogenous solution to a certain branching distributional fixed-point equation; this representation is different from \eqref{eq:LinearSFPEInd} under in-degree/out-degree correlations. Second, we compute the asymptotic tail behavior of the solution $\mathcal{R}^*$ to provide a qualitative analysis of the types of vertices that PageRank is scoring highly. Again, under in-degree/out-degree correlations this behavior is significantly different from \eqref{eq:TailLDind}. Moreover, our analysis shows that the PageRank of a randomly chosen vertex and that of its inbound neighbors can differ greatly. 

The paper is organized as follows. Section~\ref{S.RandomGraphModels} gives a brief description of the directed configuration model (DCM) and the inhomogeneous random digraph (IRD). Section~\ref{S.PageRank} provides a description of the generalized PageRank algorithm on directed graphs, as well as of its large graph limit on marked Galton-Watson processes. It also includes the first of the main theorems, which establishes the convergence of the PageRank distribution and the characterization of its limit.  Section~\ref{S.PowerLaw} includes all the results on the large deviations analysis of the solution $\mathcal{R}^*$, which represents the rank of a randomly chosen vertex, as well as of that of its inbound neighbors $\{\mathcal{R}_i\}$. To illustrate how the asymptotic analysis done on the limiting tree truly reflects the qualitative behavior of PageRank on large graphs, we include in Section~\ref{S.Numerical} some numerical experiments. Finally, Section~\ref{S.Proofs} contains all the technical proofs.

\section{Directed random graph models} \label{S.RandomGraphModels}

As mentioned in the introduction, in order to obtain a limiting distribution that can be explicitly analyzed, we need to focus on random graph models whose local structure converges to a marked Galton-Watson process, which is where solutions to branching distributional fixed-point equations like \eqref{eq:LinearSFPEInd} are constructed. Two popular random graph models with this property are the directed configuration model and the inhomogeneous random digraph. Moreover, both of these can easily be used to model scale-free real-world networks with arbitrarily dependent in-degrees and out-degrees. Recall that a scale-free graph is one whose in-degree distribution, out-degree distribution or both, follow (asymptotically) a power-law.

\subsection{Directed configuration model} \label{SS.DCM}

One model that produces graphs from any prescribed (graphical) degree sequence is the configuration or pairing model \cite{bollobas, Hofstad1}, which assigns to each vertex in the graph a number of half-edges equal to its target degree and then randomly pairs half-edges to connect vertices. 

We assume that each vertex $i$ in the graph has a degree vector ${\bf D}_i = (D_i^-, D_i^+, Q_i, \zeta_i) \in \mathbb{N}^2 \times \mathbb{R}^2$, where $D_i^-$ and $D_i^+$ are the in-degree and out-degree of vertex $i$, respectively. The values of $Q_i$ and $\zeta_i$ are not needed for drawing the graph but will be used later to compute generalized PageRank. In order for us to be able to draw the graph, we assume that the extended degree sequence $\{ {\bf D}_i: 1 \leq i \leq n\}$ satisfies
$$L_n := \sum_{i=1}^n D_i^- = \sum_{i=1}^n D_i^+.$$
Note that in order for the sum of the in-degrees to be equal to that of the out-degrees, it may be necessary to consider a double sequence $\{ {\bf D}_i^{(n)}: i \geq 1, n \geq 1\}$ rather than a unique sequence. 

Formally, the DCM can be defined as follows.

\begin{defn}
\label{D.DCM}
Let $\{ {\bf D}_i: 1 \leq i \leq n\}$ be an (extended) degree sequence and let $V_n = \{1, 2, \dots, n\}$ denote the nodes in the graph. To each node $i$ assign $D_i^-$ inbound half-edges and $D_i^+$ outbound half-edges. Enumerate all $L_n$ inbound half-edges, respectively outbound half-edges, with the numbers $\{1, 2, \dots, L_n\}$, and let ${\bf x}_n = (x_1, x_2, \dots, x_{L_n})$ be a random permutation of these $L_n$ numbers, chosen uniformly at random from the possible $L_n!$ permutations. The DCM with degree sequence $\{ {\bf D}_i: 1 \leq i \leq n\}$ is the directed graph $\mathcal{G}_n = \mathcal{G}(V_n, E_n)$ obtained by pairing the $x_i$th outbound half-edge with the $i$th inbound half-edge.
\end{defn}

We point out that instead of generating the permutation ${\bf x}_n$ of the outbound half-edges up front, one could alternatively construct the graph one vertex at a time, by pairing each of the inbound half-edges with an outbound half-edge, randomly chosen with equal probability from the set of unpaired outbound half-edges. 

We emphasize that the DCM is in general a multi-graph, that is, it can have self-loops and multiple edges in the same direction. However, provided the pairing process does not create self-loops or multiple edges, the resulting graph is uniformly chosen among all graphs having the prescribed degree sequence.  If one chooses this degree sequence according to a power-law, one immediately obtains a scale-free graph.  It was shown in~\cite{Chen_Olv_13} that the random pairing of inbound and outbound half-edges results in a simple graph with positive probability provided both the in-degree and out-degree distributions possess a finite variance. In this case, one can obtain a simple realization after finitely many attempts, a method we refer to as the {\it repeated} DCM. Furthermore, if the self-loops and multiple edges in the same direction are simply removed, a model we refer to as the {\it erased} DCM, the degree distributions will remain asymptotically unchanged.

For the purposes of this paper, self-loops and multiple edges in the same direction do not affect the main convergence result for the ranking scores, and therefore we do not require the DCM to result in a simple graph. 

We will use $\mathscr{F}_n = \sigma( {\bf D}_i: 1 \leq i \leq n)$ to denote the sigma algebra generated by the extended degree sequence, which does not include information about the random pairing. To simplify the notation, we will use $\mathbb{P}_n( \cdot ) = P( \cdot | \mathscr{F}_n)$ and $\mathbb{E}_n[ \cdot ] = E[ \cdot | \mathscr{F}_n]$ to denote the conditional probability and conditional expectation, respectively, given $\mathscr{F}_n$.

\subsection{Inhomogeneous random digraphs} \label{SS.IRD}

Alternatively, one could think of obtaining the scale-free property as a consequence of how likely different nodes are to have an edge between them. In the spirit of the classical Erd\H os-R\'enyi graph \cite{Erdos, Gilbert, Austin, Janson, Bollobas2, Durrett1}, we assume that whether there is an edge between vertices $i$ and $j$ is determined by a coin-flip, independently of all other edges. Several models capable of producing graphs with inhomogeneous degrees while preserving the independence among edges have been suggested in the recent literature, including: the Chung-Lu model \cite{Chunglu, Chunglu2, Chunglu3, Chunglu4, Lu}, the Norros-Reittu model (or Poissonian random graph) \cite{Norros, Hofstad1, Esker_Hofs_Hoog}, and the generalized random graph \cite{Hofstad1, Brittonetal, Esker_Hofs_Hoog}, to name a few. In all of these models, the inhomogeneity of the degrees is created by allowing the success probability of each coin-flip to depend on the ``attributes" of the two vertices being connected; the scale-free property can then be obtained by choosing the attributes according to a power-law.

We now give a precise description of the family of directed random graphs that we study in this paper, which includes as special cases the directed versions of all the models mentioned above. Throughout the paper we refer to a directed graph $\mathcal{G}_n = \mathcal{G}(V_n, E_n)$ on the vertex set $V_n = \{1 , 2, \dots, n\}$ simply as a {\em random digraph} if the event that edge $(i,j)$ belongs to the set of edges $E_n$ is independent of all other edges. 

In order to obtain inhomogeneous degree distributions, to each vertex $i \in V_n$ we assign a {\em type} ${\bf W}_i = (W_i^-, W_i^+, Q_i, \zeta_i) \in \mathbb{R}_+^2 \times \mathbb{R}^2$. The $W_i^-$ and $W_i^+$ will be used to determine how likely vertex $i$ is to have inbound/outbound neighbors. As for the DCM, it may be necessary to consider a double sequence $\{{\bf W}_i^{(n)}:  i \geq 1, \, n \geq 1\}$ rather than a unique sequence. With some abuse of notation, we will use $\mathscr{F}_n =\sigma( {\bf W}_i: 1 \leq i \leq n)$ to denote the sigma algebra generated by the type sequence, and  define $\mathbb{P}_n(\cdot) = P( \cdot | \mathscr{F}_n)$ and $\mathbb{E}_n[ \cdot] = E[ \cdot | \mathscr{F}_n]$ to be the conditional probability and conditional expectation, respectively, given the type sequence.

We now define our family of random digraphs using the conditional probability, given the type sequence, that edge $(i,j) \in E_n$, 
\begin{equation} \label{eq:EdgeProbabilities}
p_{ij}^{(n)} \triangleq \mathbb{P}_n \left( (i,j) \in E_n \right) = 1 \wedge  \frac{W_i^+ W_j^-}{\theta n} (1 + \varphi_n({\bf W}_i, {\bf W}_j)) , \qquad 1 \leq i \neq j \leq n,
\end{equation}
where $-1 < \varphi_n({\bf W}_i, {\bf W}_j) = \varphi(n, {\bf W}_i, {\bf W}_j, \mathscr{W}_n)$ a.s.~is a function that may depend on the entire sequence $\mathscr{W}_n := \{{\bf W}_i: 1 \leq i \leq n\}$, on the types of the vertices $(i,j)$, or exclusively on $n$, and $0<\theta < \infty$ satisfies
$$\frac{1}{n} \sum_{i=1}^n (W_i^- + W_i^+) \stackrel{P}{\longrightarrow} \theta, \qquad n \to \infty.$$
 Here and in the sequel, $x \wedge y = \min\{x,y\}$ and $x \vee y = \max\{x, y\}$. In the context of \cite{Boll_Jan_Rio_07, Cao_Olv_19}, definition \eqref{eq:EdgeProbabilities} corresponds to the so-called rank-1 kernel, i.e., $\kappa({\bf W}_i, {\bf W}_j) = \kappa_+({\bf W}_i) \kappa_-({\bf W}_j)$, with $\kappa_+({\bf W}) = W^+/\sqrt{\theta}$ and $\kappa_-({\bf W}) = W^-/\sqrt{\theta}$.

 \section{Generalized PageRank} \label{S.PageRank}
 
We now move on to the analysis of the typical behavior of the PageRank algorithm on the two directed random graph models described earlier. We show that the distribution of the ranks produced by the algorithm converges in distribution to 
a finite random variable $\mathcal{R}^*$ which can be explicitly constructed using a marked Galton-Watson process. For completeness, we give below a brief description of the algorithm, which is well-defined for any directed graph $\mathcal{G}_n = \mathcal{G}(V_n, E_n)$ on the vertex set $V_n = \{1, 2, \dots, n\}$ with edges in the set $E_n$. 

Let $D_i^-$ and $D_i^+$ denote the in-degree and out-degree, respectively, of vertex $i$ in $\mathcal{G}_n$.   The generalized PageRank vector ${\bf r} = (r_1, \dots, r_n)$ is the unique solution to the following system of equations:
\begin{equation} \label{eq:genPageRank}
r_i =   \sum\limits_{(j,i)\in E_n} \frac{\zeta_j}{D_j^+} \cdot r_j + (1-c) q_i,\quad i=1, \dots,n,
\end{equation}
where ${\bf q} = (q_1, \dots, q_n)$ is a probability vector known as the personalization or teleportation vector, $\{\zeta_i\}$ are referred to as the weights and they satisfy $|\zeta_i| \leq c$ for all $i$, and $c \in (0,1)$ is the damping factor. In the original formulation of PageRank \cite{Bri_Pag_98}, the personalization values and weights are given, respectively, by $q_i = (1-c)/n$ and $\zeta_i = c$ for all $1 \leq i \leq n$. The formulation given in \cite{Chen_Lit_Olv_17, Lee_Olv_19} is more general, and it allows any choice of personalization vector (not necessarily a probability vector) and weights. We refer the reader to \S1.1 in \cite{Chen_Lit_Olv_17} for further details on the history of PageRank, its applications, and a matrix representation of the solution ${\bf r}$ to \eqref{eq:genPageRank}. 

In order to analyze {\bf r}  on directed complex networks, we first eliminate the dependence on the size of the graph by computing the scale free ranks $(R_1, \dots, R_n) = {\bf R} =: n {\bf r}$, which corresponds to solving:
\begin{equation} \label{eq:scaleFreePageRank}
R_i =  \sum\limits_{(j,i)\in E_n}  C_j R_j +  Q_i, \quad i = 1, \dots, n,
\end{equation}
where $Q_i = (1-c) q_i n$ and $C_j = c \zeta_j/(D_j^+ \vee 1)$ (note that if vertex $j$ appears on the right-hand side of \eqref{eq:genPageRank}, then it must satisfy $D_j^+ \geq 1$, so including the maximum with one does not change the system of linear equations). In matrix notation, \eqref{eq:scaleFreePageRank} can be written as:
$${\bf R} = {\bf R} \, \text{diag}({\bf C}) {\bf A}  +  {\bf Q} ,$$
where ${\bf A}$ is the adjacency matrix of $\mathcal{G}_n$, ${\bf I}$ is the identity matrix in $\mathbb{R}^{n \times n}$, and $\text{diag}({\bf x})$ denotes the diagonal matrix defined by vector ${\bf x} = (x_1, \dots, x_n)$. It follows that the generalized PageRank vector can be written as:
$${\bf R} = {\bf Q}  ({\bf I} - {\bf M} )^{-1} =  {\bf Q}  \sum_{k=0}^\infty {\bf M}^k,$$
with ${\bf M} := \text{diag}({\bf C}) {\bf A}$. Note that the matrix ${\bf I} - {\bf M} $ is always invertible by construction, since $\norm{{\bf M}_{i \bullet}}_1 \leq c < 1$ for all $1 \leq i \leq n$, where ${\bf M}_{i\bullet}$ is the $i$th row of matrix ${\bf M}$.  In the PageRank literature it is common to replace the zero rows of matrix ${\bf M}$, which correspond to {\em dangling nodes} (vertices in $\mathcal{G}_n$ with zero out-degree), with vector ${\bf q}$ (assuming it is a probability vector). Doing this makes it possible to interpret the PageRank vector ${\bf r}$ as the stationary distribution of a random walk on graph $\mathcal{G}_n$. However, since our formulation does not require ${\bf Q}$ to be a probability vector and allows for random weights $\{\zeta_i\}$, we keep the zero rows of ${\bf M}$ intact.

\subsection{The limiting distribution} \label{SS.WeightedTree}

The work in \cite{Chen_Lit_Olv_17} and \cite{Lee_Olv_19} shows that the distribution of generalized PageRank, in both the DCM and IRD models, converges to a random variable $\mathcal{R}^*$ defined in terms of the attracting endogenous solution to a stochastic fixed-point equation known as the {\em smoothing transform}. However, the approach used there required the asymptotic independence between the in-degree and out-degree of the same vertex, which is not always realistic for modeling real-world graphs. Here, we identify a different smoothing transform that can incorporate degree-degree dependencies and still provide exact asymptotics for the generalized PageRank distribution.

In order to define the limit $\mathcal{R}^*$ to which the generalized PageRank of a randomly chosen vertex converges to, we first construct a marked delayed\footnote{``delayed'' refers to the fact that the root is allowed to have a different distribution.} Galton-Watson process. We use $\emptyset$ to denote the root node of the tree, and give every other node a label of the form ${\bf i} = (i_1, i_2, \dots, i_k) \in \mathcal{U}$, where $\mathcal{U} = \bigcup_{k=0}^\infty (\mathbb{N}_+)^k$ is the set of all finite sequences of positive integers with the convention that $\mathbb{N}_+^0 = \{ \emptyset\}$. For ${\bf i} = (i_1)$ we simply write ${\bf i} = i_1$, that is, without the parenthesis, and we use $({\bf i}, j) = (i_1, \dots, i_k, j)$ to denote the index concatenation operation.  The label of a node provides its entire lineage from the root.  Next, we use a sequence of independent vectors of the form $\{ (\mathcal{N}_{\bf i}, \mathcal{Q}_{\bf i}, \mathcal{C}_{\bf i}) : {\bf i} \in \mathcal{U} \}$, satisfying $\mathcal{N}_{\bf i} \in \mathbb{N}$ for all ${\bf i}$, to construct the tree as follows. Let $A_0 = \{ \emptyset\}$ and define 
$$A_{k} = \{ ({\bf i}, j) \in \mathcal{U}: {\bf i} \in A_{k-1}, 1 \leq j \leq \mathcal{N}_{\bf i} \}, \qquad k \geq 1,$$
to be the set of nodes at distance $k$ from the root, equivalently, the set of individuals in the $k$th generation of the tree. The vector $(\mathcal{Q}_{\bf i}, \mathcal{C}_{\bf i})$ will be referred to as the {\em mark} of node ${\bf i}$ , and we use it to define the weight $\Pi_{\bf i}$ of node ${\bf i}$ according to the recursion
$$\Pi_\emptyset \equiv 1 \qquad \text{and} \qquad \Pi_{({\bf i},j)} = \Pi_{\bf i}  \mathcal{C}_{({\bf i}, j)}, \quad {\bf i} \in \mathcal{U}.$$
Figure \ref{F.WeightedTree} illustrates the construction. 

\begin{center}
\begin{figure}[ht]
\begin{picture}(480,110)(-40,0)
\put(25,10){\includegraphics[scale = 0.75]{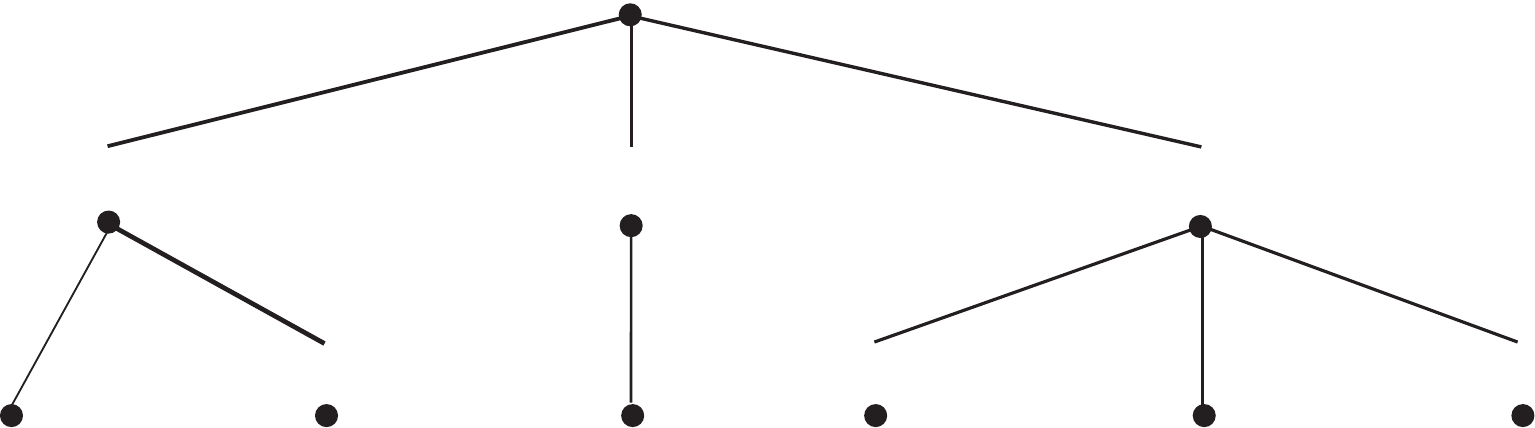}}
\put(150,107){\footnotesize $\Pi_\emptyset = 1$}
\put(30,61){\footnotesize $\Pi_{1} =  \mathcal{C}_1$}
\put(145,61){\footnotesize $\Pi_{2} =  \mathcal{C}_2$}
\put(270,61){\footnotesize $\Pi_{3} =  \mathcal{C}_3$}
\put(0,2){\footnotesize $\Pi_{(1,1)} = \Pi_1  \mathcal{C}_{(1,1)} $}
\put(62,19){\footnotesize $\Pi_{(1,2)} = \Pi_1 \mathcal{C}_{(1,2)}$}
\put(129,2){\footnotesize $\Pi_{(2,1)} =  \Pi_2  \mathcal{C}_{(2,1)}$}
\put(182,19){\footnotesize $\Pi_{(3,1)} = \Pi_3  \mathcal{C}_{(3,1)} $}
\put(252,2){\footnotesize $\Pi_{(3,2)} = \Pi_3  \mathcal{C}_{(3,2)} $}
\put(318,19){\footnotesize $\Pi_{(3,3)} = \Pi_3  \mathcal{C}_{(3,3)} $}
\end{picture}
\caption{Weighted tree.}\label{F.WeightedTree} \label{F.WBT}
\end{figure}
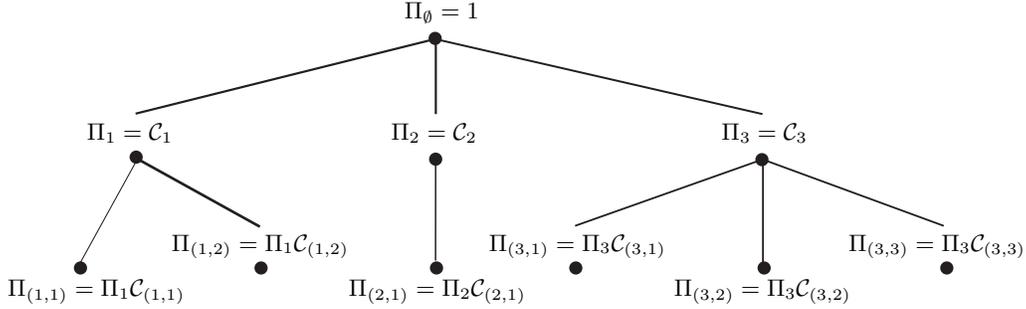
\end{center}

We will refer to $\{ (\mathcal{N}_{\bf i}, \mathcal{Q}_{\bf i}, \mathcal{C}_{\bf i}): {\bf i} \in \mathcal{U} \}$ as the sequence of branching vectors, and we will assume that they are independent, with the $\{ (\mathcal{N}_{\bf i}, \mathcal{Q}_{\bf i}, \mathcal{C}_{\bf i}): {\bf i} \in \mathcal{U}, {\bf i} \neq \emptyset \}$ i.i.d.~copies of some generic vector $(\mathcal{N}, \mathcal{Q}, \mathcal{C})$. Next, define the random variables
$$\mathcal{R}_{\bf i}  :=  \sum_{k=0}^\infty \sum_{{\bf j}: ({\bf i}, {\bf j}) \in A_{m+k}} (\Pi_{({\bf i}, {\bf j})} /\Pi_{\bf i})  \mathcal{Q}_{({\bf i},{\bf j})}, \qquad \text{for } {\bf i} \in A_m,$$
with the convention that $\Pi_{({\bf i}, {\bf j})} /\Pi_{\bf i} \equiv 1$ if $\Pi_{\bf i} = 0$. Note that these random variables satisfy
\begin{align*}
\mathcal{R}_{\bf i} &=  \mathcal{Q}_{\bf i} +   \sum_{k=1}^\infty \sum_{l=1}^{\mathcal{N}_{\bf i}}  \sum_{{\bf j}: ({\bf i},l, {\bf j}) \in A_{m+k}} (\Pi_{({\bf i},l, {\bf j})} /\Pi_{\bf i})  \mathcal{Q}_{({\bf i}, l, {\bf j})}  \\
&=  \mathcal{Q}_{\bf i} +    \sum_{l=1}^{\mathcal{N}_{\bf i}}  \mathcal{C}_{({\bf i},l)} \sum_{k=0}^\infty  \sum_{{\bf j}: ({\bf i},l, {\bf j}) \in A_{m+1+k}} (\Pi_{({\bf i},l, {\bf j})} /\Pi_{({\bf i},l)})  \mathcal{Q}_{({\bf i}, l, {\bf j})} \\
&=  \mathcal{Q}_{\bf i} + \sum_{l=1}^{\mathcal{N}_{\bf i}} \mathcal{C}_{({\bf i},l)}  \mathcal{R}_{({\bf i},l)} . 
\end{align*}
In other words, $\mathcal{R}_{\bf i}$ is the generalized PageRank of node ${\bf i}$ in the marked tree. 

If we assumed independence between $\mathcal{C}_{\bf i}$ and $\mathcal{R}_{\bf i}$, we could relate the $\{\mathcal{R}_{\bf i}\}$ with the attracting endogenous solution to the stochastic fixed-point equation
$$\mathcal{R} \stackrel{\mathcal{D}}{=}  \mathcal{Q} +  \sum_{j=1}^{\mathcal{N}} \mathcal{C}_j \mathcal{R}_j,$$
where the $\{\mathcal{R}_j\}$ are i.i.d.~copies of $\mathcal{R}$ and are independent of the vector $(\mathcal{Q}, \mathcal{N},  \{ \mathcal{C}_i\})$. However, the degree-degree correlations considered here violate the independence between $\mathcal{C}_j$ and $\mathcal{R}_j$, and the distributional equation above no longer holds. To fix this problem, define
$$X_{\bf i} := \mathcal{C}_{\bf i} \mathcal{R}_{\bf i}$$
and note that the $\{X_{\bf i}\}_{{\bf i} \in \mathcal{U}, {\bf i} \neq \emptyset}$ satisfy
$$X_{\bf i} =  \mathcal{C}_{\bf i} \mathcal{Q}_{\bf i} + \sum_{j=1}^{\mathcal{N}_{\bf i}}    \mathcal{C}_{\bf i}  X_{({\bf i}, j)},$$
with the $\{ X_{({\bf i}, j)}\}_{j \geq 1}$ i.i.d.~and independent of $(\mathcal{Q}_{\bf i}, \mathcal{N}_{\bf i}, \mathcal{C}_{\bf i})$.  To see this independence note that for ${\bf i} \in A_{m+1}$ we have
\begin{align*}
X_{({\bf i}, j)} &=  \sum_{k=0}^\infty \sum_{{\bf l}: ({\bf i},j, {\bf l}) \in A_{m+1+k}} \mathcal{C}_{({\bf i}, j)} (\Pi_{({\bf i},j, {\bf l})} /\Pi_{({\bf i},j)} )  \mathcal{Q}_{({\bf i},j, {\bf l})} ,
\end{align*}
with $\Pi_{({\bf i}, j, {\bf l})} / \Pi_{({\bf i}, j)}$ independent of $(\mathcal{Q}_{\bf i}, \mathcal{N}_{\bf i}, \mathcal{C}_{\bf i})$ for any ${\bf l} \in \mathcal{U}$. In other words, the $\{  X_{\bf i} \}$ satisfy the stochastic fixed-point equation
\begin{equation} \label{eq:NewSFPE}
X \stackrel{\mathcal{D}}{=} \mathcal{C} \mathcal{Q} + \sum_{j=1}^\mathcal{N} \mathcal{C} X_j,
\end{equation}
where the $\{X_j\}$ are i.i.d.~copies of $X$ and are independent of the vector $(\mathcal{Q}, \mathcal{N}, \mathcal{C})$. Moreover, the $\{X_j\}$ have the same distribution as the attracting endogenous solution to \eqref{eq:NewSFPE}, which has been extensively analyzed in \cite{Holl_Ligg_81, Durr_Ligg_83, Way_Will_95, Als_Big_Mei_10, Alsm_Mein_12, Alsm_Mein_13, Jel_Olv_12a, Jel_Olv_12b, Olvera_12b, Als_Dam_Men_12, Bur_Dam_Zie_15,Alsm_Dysz_15} .  

Finally, the random variable $\mathcal{R}^*$ to which the generalized PageRank of a randomly chosen vertex converges to, can be written as:
$$\mathcal{R}^* = \mathcal{Q}_0 + \sum_{j=1}^{\mathcal{N}_0} X_j,$$
where the $\{X_j\}_{j \geq 1}$ are i.i.d., independent of $(\mathcal{Q}_0, \mathcal{N}_0)$, and have the same distribution as $X$.  Note that the vector $(\mathcal{Q}_0, \mathcal{N}_0)$ is allowed to have a different distribution than the generic branching vector $(\mathcal{Q}, \mathcal{N}, \mathcal{C})$, which is important in view of the bias produced by sampling vertices in the graph who are already known to have an outbound neighbor.

We now give the main assumption needed for the convergence of the generalized PageRank distribution. For the DCM, let
\begin{equation} \label{eq:DegreeDistr}
F_n(m,k,q,x) = \frac{1}{n} \sum_{i=1}^n 1(D_i^- \leq m, D_i^+ \leq k, Q_i \leq q, \zeta_i \leq x),
\end{equation}
and for the IRD, let
\begin{equation} \label{eq:TypeDistr}
F_n(u,v,q,x) = \frac{1}{n} \sum_{i=1}^n 1(W_i^- \leq u, \, W_i^+ \leq v, \, Q_i \leq q, \, \zeta_i \leq x).
\end{equation}
Our main convergence assumption is given in terms of the Kantorovich-Rubinstein distance (or Wasserstein distance of order 1), denoted $d_1$.

\begin{assumption} \label{A.Primitives}
Let $F_n$ be defined according to either \eqref{eq:DegreeDistr} or \eqref{eq:TypeDistr}, depending on the model, and suppose there exists a distribution $F$ (different for each model) such that
$$d_1(F_n, F) \stackrel{P}{\longrightarrow} 0, \qquad n \to \infty.$$
In addition, assume that $\max_{1 \leq i \leq n} |\zeta_i| \leq 1$ and the following conditions hold:
\begin{itemize}
\item[A.] In the DCM, let $(\mathscr{D}^-, \mathscr{D}^+, Q, \zeta)$ have distribution $F$, and suppose the following hold:
\begin{enumerate}
\item[1.] $E[ \mathscr{D}^-] = E[\mathscr{D}^+]$.

\item[2.] $E[ \mathscr{D}^- + \mathscr{D}^+  + |Q| ] < \infty$ and $|\zeta| \leq c < 1$  a.s.
\end{enumerate}

\item[B.] In the IRD, let $(W^-, W^+, Q, \zeta)$ have distribution $F$, and suppose the following hold:
\begin{enumerate}
\item[1.] $\displaystyle \mathcal{E}_n = \frac{1}{n} \sum_{i=1} \sum_{1 \leq j \leq n, j \neq i} |p_{ij}^{(n)} - (r_{ij}^{(n)} \wedge 1) | \xrightarrow{P} 0$ as $n \to \infty$, where $r_{ij}^{(n)} = W_i^+ W_j^-/(\theta n)$.

\item[2.] $E[ W^- + W^+  + |Q| ] < \infty$ and $|\zeta| \leq c < 1$  a.s.

\end{enumerate}

\end{itemize}
\end{assumption}

We now give the main convergence result for the PageRank of a randomly chosen vertex to $\mathcal{R}^*$. This result fully generalizes Theorem~6.4 in \cite{Chen_Lit_Olv_17} for the DCM and Theorem~3.3 in \cite{Lee_Olv_19} for the IRD, respectively. Specifically, it holds under weaker moment conditions, it allows the in-degree and out-degree of each vertex to be arbitrarily dependent, and shows a stronger mode of convergence which implies the convergence of the means\footnote{Theorem~6.4 in \cite{Chen_Lit_Olv_17} and Theorem~3.3 in \cite{Lee_Olv_19} give only the convergence in distribution.}.

\begin{theo} \label{T.PageRank}
Suppose that one of the following holds:
\begin{itemize}
\item[i)] The graph $\mathcal{G}_n$ is a DCM and its extended degree sequence satisfies Assumption~\ref{A.Primitives}(A).
\item[ii)] The graph $\mathcal{G}_n$ is an IRD and its type sequence satisfies Assumption~\ref{A.Primitives}(B). 
\end{itemize}
Let $R_\xi$ denote the rank of a uniformly chosen vertex in $\mathcal{G}_n$ and let $H_n(x) = \mathbb{P}_n( R_\xi \leq x)$. Then, 
$$d_1(H_n, H) \xrightarrow{P} 0, \qquad n \to \infty,$$
where $H(x) = P( \mathcal{R}^* \leq x)$, $E[|\mathcal{R}^*|] < \infty$, and 
\begin{equation} \label{eq:LinearCombination}
\mathcal{R}^*= \mathcal{Q}_0 +  \sum_{j=1}^{\mathcal{N}_0} X_j,
\end{equation}
where the $\{ X_j \}_{j \geq 1}$ are i.i.d.~copies of the attracting endogenous solution to \eqref{eq:NewSFPE}, independent of $(\mathcal{N}_0, \mathcal{Q}_0)$. Moreover, the vectors $(\mathcal{N}_0, \mathcal{Q}_0)$ and  $(\mathcal{N}, \mathcal{Q}, \mathcal{C})$ satisfy for $m \in \mathbb{N}$, $t,q \in \mathbb{R}$,
\begin{align*}
P( \mathcal{N}_0 = m, \mathcal{Q}_0 \in dq) &= P( \mathscr{D}^- = m, Q \in dq) \\
P( \mathcal{N} = m, \mathcal{Q} \in dq, \mathcal{C} \in dt) &= E\left[ 1( \mathscr{D}^- = m, Q \in dq, \zeta/(\mathscr{D}^+ \vee 1) \in dt) \cdot \frac{\mathscr{D}^+}{E[\mathscr{D}^+]} \right]
\end{align*}
for the DCM and
\begin{align*}
P( \mathcal{N}_0 = m, \mathcal{Q}_0 \in dq) &= P\left( Z^- = m, Q \in dq \right) \\
P( \mathcal{N} = m, \mathcal{Q} \in dq, \mathcal{C} \in dt) &= E\left[1(Z^- = m,  Q \in dq, \zeta/(Z^++1) \in dt) \cdot \frac{W^+ }{E[W^+]} \right],
\end{align*}
for the IRD, where $Z^-$ and $Z^+$ are mixed Poisson random variables with mixing parameters $E[W^+] W^-/\theta$ and $E[W^-] W^+/\theta$, respectively, conditionally independent given $(W^-, W^+)$.  
\end{theo}

\begin{remark}
Note that the theorem does not preclude the possibility that $E[ \mathcal{N} + |\mathcal{Q}|] = \infty$, which will indeed be the case whenever $E[ \mathscr{D}^- \mathscr{D}^+ + |Q|  \mathscr{D}^+] = \infty$ in the DCM or when $E[ W^- W^+ + |Q|  W^+] = \infty$ in the IRD. In fact, if $E[ \mathcal{N} ] = \infty$, the underlying weighted branching process on which the $\{X_j\}$ are constructed will have an infinite mean offspring distribution. What is interesting, is that even if this is the case, the distribution of PageRank will remain well-behaved, and in particular, will have finite mean. 
\end{remark}


\section{The power law behavior of PageRank} \label{S.PowerLaw}

Using the characterization of $\mathcal{R}^*$ given in Theorem~\ref{T.PageRank}, we can now establish the power law behavior of PageRank on a graph $\mathcal{G}_n$ whenever the limiting in-degree distribution of the graph, i.e., $\mathcal{N}_0$, has a power law distribution. The theorem below gives two possible asymptotic expressions depending on the distribution of the generic branching vector $(\mathcal{Q}, \mathcal{N}, \mathcal{C})$. In the statement of the theorem, we have made no assumptions on the relationship between the distributions of $(\mathcal{N}_0, \mathcal{Q}_0)$ and $(\mathcal{N}, \mathcal{Q}, \mathcal{C})$, i.e., they do not require to be related via the distributions in Theorem~\ref{T.PageRank}.  Throughout this section we assume that $\mathcal{C} \geq 0$ and use the notation $\rho_\alpha := E[ \mathcal{N} \mathcal{C}^\alpha]$ for $\alpha > 0$. We also need the following definitions.

\begin{defn}
Let $X$ be a random variable with right tail distribution $\overline{F}(x) = P(X > x)$. 
\begin{itemize}
\item We say that $\overline{F}$ is {\em regularly varying with tail index $\alpha$}, denoted $\overline{F} \in RV(-\alpha)$, if
$$\lim_{x \to \infty} \frac{\overline{F}(\lambda x)}{\overline{F}(x)} = \lambda^{-\alpha} \quad \text{for any $\lambda > 0$}. $$

\item We say that $\overline{F}$ belongs to the {\em intermediate regular variation} class, denoted $\overline{F} \in IR$, if
$$\lim_{\delta \downarrow 0} \limsup_{x \to \infty} \frac{\overline{F}((1-\delta) x)}{\overline{F}(x)} = 1.$$
\end{itemize}
\end{defn}

Our first theorem describing the asymptotic behavior of $\mathcal{R}^*$ is given below. Throughout the paper we write $f(x) = o(g(x))$ as $x \to \infty$ whenever $\lim_{x \to \infty} f(x)/g(x) = 0$. 

\begin{theo} \label{T.PageRankRoot} 
Fix $\alpha > 1$ and suppose $\rho_1 \vee \rho_\alpha < 1$. Then,
\begin{itemize}
\item[a.)] If $P(\mathcal{N C} > x) \in RV(-\alpha)$, $E[ \mathcal{Q C}] > 0$, $E[|\mathcal{Q C}|^{\alpha+\epsilon}] < \infty$ and $\rho_{\alpha+\epsilon} < \infty$, for some $\epsilon > 0$, we have
$$P(X > x) \sim \frac{(E[\mathcal{Q C}])^\alpha}{(1-\rho_1)^\alpha (1-\rho_\alpha)} P( \mathcal{NC} > x), \qquad x \to \infty.$$
If in addition, $P(\mathcal{N}_0 > x) \in IR$, $E[\mathcal{N}_0] < \infty$ and $P(\mathcal{Q}_0 > x) = o(P(\mathcal{N}_0 > x))$ as $x \to \infty$, we have 
$$P(\mathcal{R}^* > x) \sim E[\mathcal{N}_0] P( X > x) + P\left( \mathcal{N}_0 > x/E[X] \right), \qquad x \to \infty. $$

\item[b.)] If $P(\mathcal{Q C} > x) \in RV(-\alpha)$, $E[|\mathcal{Q C}|^\beta] < \infty$ for all $0 < \beta < \alpha$,  and $E[ (\mathcal{N} \mathcal{C})^{\alpha+\epsilon}] < \infty$ for some $\epsilon > 0$, we have 
$$P(X > x) \sim (1- \rho_\alpha)^{-1} P( \mathcal{QC} > x), \qquad x \to \infty.$$
If in addition, $P(\mathcal{Q}_0 > x) \in IR$, $E[\mathcal{N}_0] < \infty$, and $P(\mathcal{N}_0 > x) = o(P(\mathcal{Q}_0 > x))$ as $x \to \infty$, we have 
$$P(\mathcal{R}^* > x) \sim E[ \mathcal{N}_0 ] P(X > x) + P( \mathcal{Q}_0 > x), \qquad  x \to \infty. $$
\end{itemize}
\end{theo}

\begin{proof}
The results for the asymptotic behavior of $P(X > x)$ are a direct application of Theorems~3.4 and 4.4 in \cite{Olvera_12b}. 
The results for $P(\mathcal{R}^* > x)$ follow from Theorems~\ref{T.SumFatM} and \ref{T.SumFatY} in Section~\ref{SS.PowerLaw}.
\end{proof}

For the two random graph models we study here, we have that if either $(\mathscr{D}^- , Q)$ and $(\mathscr{D}^+, \zeta)$ are independent in the DCM or if $(W^-, Q)$ and $(W^+, \zeta)$ are independent in the IRD, then  $(\mathcal{N}, \mathcal{Q}) \stackrel{\mathcal{D}}{=} (\mathcal{N}_0, \mathcal{Q}_0)$. Furthermore, if either $P(\mathcal{N}_0 > x) \in RV(-\alpha)$ in part (a) of Theorem~\ref{T.PageRankRoot} or $P(\mathcal{Q}_0 > x) \in RV(-\alpha)$ in part (b),  then either $P(\mathcal{NC} > x) \sim E[\mathcal{C}^\alpha] P(\mathcal{N}_0 > x)$ or $P(\mathcal{QC} > x) \sim E[\mathcal{C}^\alpha] P(\mathcal{Q}_0 > x)$, respectively, as $x \to \infty$ (by Breiman's Theorem). In particular,
$$E[\mathcal{N}_0] P(X > x)  \sim E[\mathcal{N}_0] E[\mathcal{C}^\alpha] P(\mathcal{R} > x) \sim E[\mathcal{C}^\alpha] P\left( \max_{1 \leq i \leq \mathcal{N}_0}  \mathcal{R}_i > x \right), \qquad x \to \infty.$$
We can then rewrite the asymptotics for $P(\mathcal{R}^* > x)$ as:
$$P( \mathcal{R}^* > x) \sim E[\mathcal{C}^\alpha] P\left( \max_{1 \leq i \leq \mathcal{N}_0} \mathcal{R}_i > x \right) + P( \mathcal{N}_0 > x/(E[\mathcal{C}] E[\mathcal{R}])), \qquad x \to \infty,$$
in part (a) or
$$P( \mathcal{R}^* > x) \sim  E[\mathcal{C}^\alpha] P\left( \max_{1 \leq i \leq \mathcal{N}_0} \mathcal{R}_i > x \right) + P( \mathcal{Q}_0 > x), \qquad x \to \infty.$$
in part (b). These expressions can then be interpreted as follows for each of the two cases:
\begin{itemize}
\item[a.)] {\em Most likely, vertices with very high ranks have either a very highly ranked inbound neighbor, or have a very large number of (average-sized) neighbors.}
\item[b.)] {\em Most likely, vertices with very high ranks have either a very highly ranked inbound neighbor, or have a very large personalization value.}
\end{itemize}

What is interesting, is that when $\mathcal{C}$ and $(\mathcal{N}, \mathcal{Q})$ are allowed to be dependent, it is possible to disappear the contribution of the highly ranked inbound neighbor whenever 
$$P(X > x) = o\left( P(\mathcal{N}_0 > x) + P(\mathcal{Q}_0 > x) \right), \quad x \to \infty,$$
 in other words, PageRank will be mostly determined by the in-degree or personalization value of each vertex. 
 
 \begin{example} \label{E.CorrelatedRV}
Suppose that either $P(\mathscr{D}^- > x) \in RV(-\alpha)$, $\alpha > 1$, and
 $$P( \mathscr{D}^-/(\mathscr{D}^+ \vee 1) > x) = o(P(\mathscr{D}^- > x)), \qquad x \to \infty,$$
 in the DCM, or $P(W^- > x) \in RV(-\alpha)$, $\alpha > 1$, and 
 $$P( W^-/(W^+\vee 1) > x) = o(P(W^- > x)), \qquad x \to \infty,$$
 in the IRD. Then, we claim that
 $$P(\mathcal{CN} > x) = o\left( P(\mathcal{N}_0 > x) \right), \qquad x \to \infty.$$
 Similarly, by replacing $\mathscr{D}^-$ and $W^-$ with $Q$ in the conditions stated above, we claim that
 $$P(\mathcal{CQ} > x) = o\left( P(\mathcal{Q}_0 > x) \right), \qquad x \to \infty.$$
 The proof of these claims can be found in Section~\ref{SS.PowerLaw}.
 \end{example}

The dependence between $\mathcal{C}$ and $(\mathcal{N}, \mathcal{Q})$ can also introduce an important bias in the PageRank of vertices in the in-component of the randomly chosen vertex, as the following theorem illustrates. 

Note that the PageRanks of vertices encountered through the exploration of a randomly chosen vertex have the distribution of the $\{ \mathcal{R}_j\}$ in:
\begin{equation} \label{eq:RootSFPE}
\mathcal{R}^* =  \mathcal{Q}_0 + \sum_{j=1}^{\mathcal{N}_0} \mathcal{C}_j \mathcal{R}_j =  \mathcal{Q}_0 + \sum_{j=1}^{\mathcal{N}_0} X_j.
\end{equation}

\begin{theo} \label{T.PageRankNeighbors}
Let $\mathcal{R}$ denote a random variable having the same distribution as the $\{ \mathcal{R}_j\}$ in \eqref{eq:RootSFPE}. 
Fix $\alpha > 1$ and suppose $\rho_1 \vee \rho_\alpha < 1$. 
\begin{itemize}
\item[a.)] Assume $f(x) = P(\mathcal{N} > x) \in IR$,  $P(\mathcal{N C} > x) \in RV(-\alpha)$, $E[ \mathcal{Q C}] > 0$, $E[|\mathcal{Q C}|^{\alpha+\epsilon}] < \infty$ and $\rho_{\alpha+\epsilon} < \infty$, for some $\epsilon > 0$, and $P(\mathcal{Q} > x) = o(P(\mathcal{N} > x))$ as $x \to \infty$. If $E[\mathcal{N}] = \infty$ assume further that $f$ has Matuszewska indexes $\alpha(f), \beta(f)$ satisfying $-(\alpha \wedge 2) < \beta(f) \leq \alpha(f) < 0$. Then,
$$P(\mathcal{R} > x) \sim 1(E[\mathcal{N}] < \infty) E[\mathcal{N}] P( X > x) + P(\mathcal{N} > x/E[X]), \qquad x \to \infty.$$

\item[b.)] Assume $g(x) = P(\mathcal{Q} > x) \in IR$, $P(\mathcal{Q C} > x) \in RV(-\alpha)$, $E[|\mathcal{Q C}|^\beta] < \infty$ for all $0 < \beta < \alpha$,  $E[ (\mathcal{N} \mathcal{C})^{\alpha+\epsilon}] < \infty$ for some $\epsilon > 0$,  and $P(\mathcal{N} > x) = o(P(\mathcal{Q} > x))$ as $x \to \infty$.  If $E[\mathcal{N}] = \infty$ assume further that $g$ has Matuszewska indexes $\alpha(g), \beta(g)$ satisfying $-(\alpha \wedge 2) < \beta(g) \leq \alpha(g) < 0$. Then, 
$$P(\mathcal{R} > x) \sim 1(E[\mathcal{N}] < \infty) E[ \mathcal{N} ] P(X > x) + P( \mathcal{Q} > x), \qquad  x \to \infty. $$
\end{itemize}
\end{theo}

\begin{proof}
The results are a direct consequence of the results for $P(X > x)$ in Theorem~\ref{T.PageRankRoot} combined with Theorems~\ref{T.SumFatM} and \ref{T.SumFatY} in Section~\ref{SS.PowerLaw}.
\end{proof}

\begin{remark}
As mentioned earlier, we have that if $(\mathscr{D}^-, Q)$ and $(\mathscr{D}^+, \zeta)$ are independent in the DCM, or if $(W^-, Q)$ and $(W^+, \zeta)$ are independent in the IRD, then $(\mathcal{N}, \mathcal{Q}) \stackrel{\mathcal{D}}{=} (\mathcal{N}_0, \mathcal{Q}_0)$. Moreover, one can verify that the asymptotic expressions in Theorems~\ref{T.PageRankRoot} and \ref{T.PageRankNeighbors} reduce to the known results from \cite{Chen_Lit_Olv_17} and \cite{Lee_Olv_19}, i.e., to having $\mathcal{R}^* \stackrel{\mathcal{D}}{=} \mathcal{R}$ and
$$\mathcal{R} \stackrel{\mathcal{D}}{=} \mathcal{Q} + \sum_{i=1}^{\mathcal{N}} \mathcal{C}_i \mathcal{R}_i,$$
withe $\{ \mathcal{R}_i\}$ i.i.d.~copies of $\mathcal{R}$, independent of $(\mathcal{Q}, \mathcal{N}, \{ \mathcal{C}_i\})$, and the $\{\mathcal{C}_i\}$ i.i.d.~and independent of $(\mathcal{Q}, \mathcal{N})$. In other words, the size-bias disappears.
\end{remark}

\begin{remark}
As is the case in the analysis of undirected random graphs, the size bias encountered while exploring vertices in the in-component of another vertex can cause the distribution of their PageRanks to be up to one moment heavier than that of a randomly chosen vertex. Hence, the bias can be quite significant and needs to be taken into account when analyzing samples obtained by following outbound links/arcs. 
\end{remark}

\section{Numerical Experiments} \label{S.Numerical}

In order to illustrate how the qualitative insights derived from Theorem~\ref{T.PageRankRoot} accurately describe the typical behavior of PageRank in large random graphs, we simulated a graph with $n = 10,000$ vertices and two choices of joint degree distributions, one with independent in-degree and out-degree and one with positive in-degree/out-degree correlation. 
Since this experiment is meant only to illustrate the qualitative differences between the two scenarios, we include only experiments done on an IRD (the corresponding results for the DCM where essentially identical). Although our experiments involve only the original PageRank algorithm (i.e., $C_i = d/(D_i^+ \vee 1)$ and $Q_i = 1-d$) for which the in-degree dominates the personalization value (part (a) of Theorem~\ref{T.PageRankRoot}), similar experiments could be done for the opposite case (part (b) of Theorem~\ref{T.PageRankRoot}).

In the experiments, the marginal in-degree and out-degree distributions where chosen to be the same in both scenarios, which for simplicity we chose to have Pareto {\em type} distributions, i.e., 
$$P(W^- > x) = (x/b)^{-\alpha}, \qquad x \geq b$$
for the in-degree parameter and 
$$P(W^+ > x) = (x/c)^{-\beta}, \qquad x \geq c$$
for the out-degree parameter; recall that the limiting degrees in the IRD will be mixed Poisson with parameters $E[W^+] W^-/\theta$ for the in-degree and $E[W^-] W^+/\theta$ for the out-degree, where $\theta = E[W^- + W^+]$. For the independent case, $W^-$ and $W^+$ were taken to be independent (which guarantees the independence of the corresponding in-degree and out-degree), and for the dependent case we set
$$W^+ = c(W^-/b)^{\alpha/\beta},$$
which gives a covariance between the in-degree and out-degree of:
$$\cov(Z^-, Z^+) = \frac{E[W^+] E[W^-]}{\theta^2} \cov (W^-, W^+) > 0. $$
For this choice of mixing distributions we have
$$E[Z^- ] = E[Z^+] = \frac{E[W^-] E[W^+]}{\theta} = \frac{bc \alpha \beta}{b\alpha(\beta-1) + c\beta(\alpha-1)}$$
in both cases. We set the parameters $(\alpha, b) = (1.5,8)$ and $(\beta, c) = (2.5, 12)$ to obtain an average degree $E[Z^-] = E[Z^+] = 10.91$.

To generate the graphs, we sampled $n$ i.i.d.~copies of $(W^-, W^+)$ and we used the edge probabilities:
$$p_{ij}^{(n)} = 1 \wedge  \frac{W_i^+ W_j^-}{\theta n}, \qquad 1 \leq i \neq j \leq n,$$
which corresponds to the directed Chung-Lu model. Once we had generated the two graphs (one where $(W^-, W^+)$ are independent and one where they are positively correlated), we computed their corresponding scale-free PageRank vector ${\bf R}$, as  defined via \eqref{eq:scaleFreePageRank} with $C_i = d/(D_i^+ \vee 1)$, $Q_i = 1-d$ for each $1 \leq i \leq n$, and $d = 0.85$ (a standard choice for the damping factor). To compute ${\bf R}$ we used matrix iterations, specifically, we used the approximation ${\bf R} \approx {\bf Q} \sum_{k=0}^m {\bf M}^k$ for $m = 30$.

Once all the scale-free PageRank scores had been computed, we took the top 5\% ranked vertices in each graph (for $n = 10,000$, the top 500 vertices) and used them to create the set $A$ of large PageRank vertices. Similarly, we created the set $B$ of large in-degree vertices by taking the top 5\%; in this case, since ties are not that rare, the exact number of vertices could be slightly larger than 500. Finally, we created the set $C$ by first identifying the top 5\% of rank-contributing vertices, i.e., those whose contribution $C_i R_i$ to their outbound neighbors is large, call this set $H$, and then 
selecting all the vertices who had at least one inbound neighbor in the set $H$ ($C = \{i \in \mathcal{G}_n: \text{$i$ has an inbound neighbor in $H$}\}$). Figure~\ref{F.VennDiagrams} depicts the relative sizes and the relationship among the sets $A$, $B$ and $C$ for the two cases after taking the average over 20 independent realizations of the entire experiment. The exact numerical values are given in Table~\ref{T.VennDiagramsData}. 

\begin{figure}
\centering
\begin{subfigure}[h]{.48\textwidth}
\centering
\includegraphics[scale=0.7, bb = 90 415 420 675, clip]{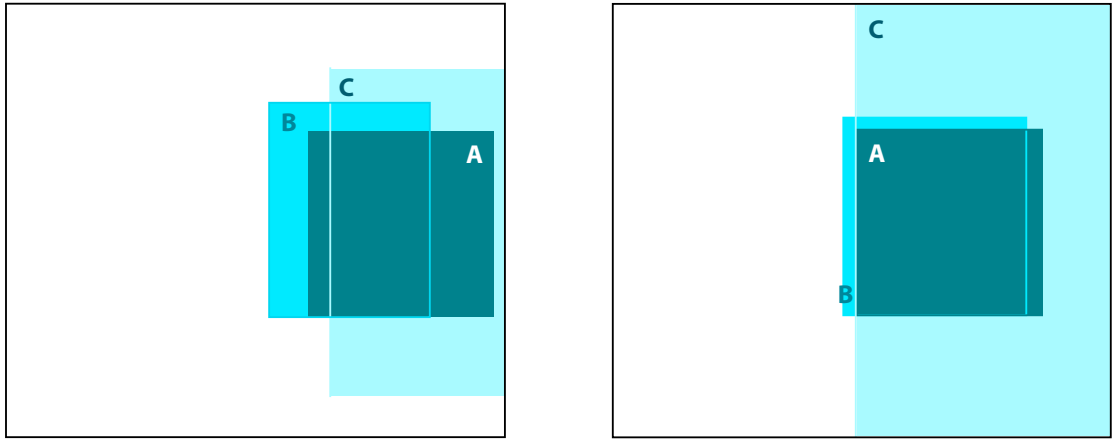}
\caption{Independent in-degree and out-degree.}
 \label{F.VennIND}
\end{subfigure}
\begin{subfigure}[h]{.48\textwidth}
\centering
\includegraphics[scale=0.7, bb = 90 110 420 370, clip]{VennDiagrams.pdf}
\caption{Dependent in-degree and out-degree.}\label{F.VennDEP}
\end{subfigure}
\caption{Qualitative relationship among the sets $A =$ \{Vertices with top 5\% PageRank scores\}, $B=$ \{Vertices with top 5\% in-degrees\}, and $C=$ \{Vertices having at least one inbound neighbor in the set $H$\}, where $H =$ \{Vertices with top 5\% values $R_i C_i$\}.}\label{F.VennDiagrams}
\end{figure}

\begin{table}[h]
\centering
\begin{tabular}[h]{| c | c | c|}
\hline & {\bf Independent case} & {\bf Dependent case} \\
{\bf Sets}  & (\%) & (\%) \\ \hline
$A \cap B \cap C$ & 3.1 & 4.59 \\
$A \cap B \cap C^c$ & 0.25  & 0.01 \\
$A^c \cap B \cap C$ & 0.82 & 0.52 \\
$A \cap B^c \cap C$ & 1.65 & 0.4 \\
$A \cap B^c \cap C^c$ & 0 & 0 \\
$A^c \cap B \cap C^c$ & 0.98 & 0.02 \\
$A^c \cap B^c \cap C$ & 16.7 & 54.39 \\
$(A \cup B \cup C)^c$ & 76.5 & 40.07 \\ \hline
$A \cap H$ & 3.43 & 1.04 \\
$A \cap H^c$ & 1.57 & 3.96 \\
$A^c \cap H$ & 1.57 & 3.96 \\ \hline
\end{tabular} 
\caption{Average percentage of vertices in the sets $A$, $B$, $C$ and $H$ after 20 independent realizations of the experiment.} \label{T.VennDiagramsData}
\end{table}

As we can see, the experiments clearly show that the relationship among the three sets $A$, $B$ and $C$ is qualitatively different between the independent and dependent cases. In both cases we have that $A \subseteq B \cup C$ (highly ranked vertices either have large in-degrees or a highly-contributing inbound neighbor), however, in the independent case we almost have equality between the sets $A$ and $B$, that is, the top ranked vertices almost coincide with the top in-degree vertices, as Theorem~\ref{T.PageRankRoot} suggests. Another important difference lies in the size of the set $C$ of vertices having a highly-contributing inbound neighbor. In the independent case, this set is relatively small (22.27\% of all the vertices) and has about 2/3 of them outside $A \cup B$  ($C \setminus (A \cup B)$ has 16.7\%); however, in the dependent case, the set $C$ is not small (59.9\% of the vertices), and it has 9/10 of its vertices outside $A \cup B$ ($C \setminus (A \cup B)$ has 54.39\%).  In other words, compared to the independent case, when the in-degree and out-degree are dependent fewer vertices having highly-contributing inbound neighbors achieve high ranks. This is explained by the observation that when the in-degree and out-degree are independent, there are highly ranked vertices with average sized out-degrees that contribute greatly to the rank of their outbound neighbors, while in the dependent case many highly ranked vertices also have large out-degrees, and their contribution to the ranks of their neighbors is largely diminished.  We can clearly see this phenomenon by looking at the last three rows of Table~\ref{T.VennDiagramsData}, which show that the intersection between the highly ranked vertices, set $A$, and those that contribute highly to the ranks of their outbound neighbors, set $H$, is very different in the two cases, with the intersection being more than three times larger in the independent case. 

To summarize, the experiments show that the predictions based on the asymptotic analysis of $P(\mathcal{R}^* > x)$ in Theorem~\ref{T.PageRankRoot} do indeed hold for the actual PageRank scores in a random graph generated through the IRD model. As mentioned earlier, similar experiments done for the DCM and for the case when $\mathcal{Q}_0$ is heavier than $\mathcal{N}_0$ show the same agreement with the theory, illustrating the valuable insights that the theory provides.

\section{Proofs} \label{S.Proofs}

To organize the proofs we have divided them into two subsections, one that includes all the results needed to prove Theorem~\ref{T.PageRank} and one that includes the proofs for all the results presented in Section~\ref{S.PowerLaw}, which are mostly related to the power-law behavior of the limiting $\mathcal{R}^*$. 

\subsection{Proof of Theorem~\ref{T.PageRank}}

The proof of Theorem~\ref{T.PageRank} is based on a coupling between the rank of a randomly chosen vertex in $\mathcal{G}_n$ and the rank of the root node of a marked (delayed) Galton-Watson process. The proof of the coupling for the IRD has been mostly done in \cite{Lee_Olv_19} under the same conditions included here (however, the analysis of the coupled tree under degree-degree dependance is new). For the DCM the existing proof in \cite{Chen_Lit_Olv_17} requires more moment conditions than those in Theorem~\ref{T.PageRank}, so for completeness we include a short proof here.

\subsubsection{Coupling the graph with a marked branching process}

The branching processes used in the coupling for the two models are slightly different. For the DCM the coupling is done with a (delayed) marked Galton-Watson process whose degree/mark distribution is determined by the extended degree sequence $\{ (D_i^-, D_i^+, Q_i, \zeta_i): 1 \leq i \leq n \}$, while for the IRD model it is determined by $\{ (W_i^-, W_i^+, Q_i, \zeta_i): 1 \leq i \leq n\}$. 

For the IRD model, the coupling result was recently done in \cite{Lee_Olv_19} (see Theorem~3.7) under the same conditions used in this paper, although only for the convergence in distribution, not the convergence in mean. Before writing down the distribution of the coupled tree we first introduce some notation. For sequences  $a_n \to \infty, b_n \to \infty$ such that $a_n b_n/n \to 0$ as $n \to \infty$, define 
$$(\bar W_i^-, \bar W_i^+) = (W_i^- \wedge a_n, W_i^+ \wedge b_n), \qquad 1 \leq i \leq n$$
and
$$\Lambda_n^- = \sum_{i=1}^n \bar W_i^-, \qquad \Lambda_n^+ = \sum_{i=1}^n \bar W_i^+.$$
The coupled (delayed) Galton-Watson process has offspring/mark joint distributions
\begin{align}
&\mathbb{P}_n\left( \hat N_\emptyset = m, \hat Q_\emptyset = q \right) = \sum_{s=1}^n 1(Q_s = q) p(m; \Lambda_n^+ \bar W_s^-/(\theta n)) \cdot \frac{1}{n} \qquad \text{and} \label{eq:RootIRDtree} \\
&\mathbb{P}_n\left( \hat N_{\bf i} = m, \hat Q_{\bf i} = q, \hat C_{\bf i} = t \right)  \notag \\
&= \sum_{s=1}^n 1(Q_s = q, c\zeta_s/t -1 \in \mathbb{N}) \cdot p(m; \Lambda_n^+ \bar W_s^-/(\theta n))\cdot p(\zeta_s/t-1; \Lambda_n^- \bar W_s^+/(\theta n)) \cdot \frac{\bar W_s^+}{\Lambda_n^+}, \qquad {\bf i} \neq \emptyset, \label{eq:NodesIRDtree}
\end{align}
for $m \in \mathbb{N}$, $q \in \mathbb{R}$, $t \in [-c, c]$, where $p(m; \lambda) = e^{-\lambda} \lambda^m/m!$ is the Poisson probability mass function with mean $\lambda$.  The coupling in \cite{Lee_Olv_19} provides an exploration of the in-component of a randomly chosen vertex in $\mathcal{G}_n$ as well as the simultaneous construction of the delayed marked tree. In particular, it defines a stopping time $\tau$ that determines the number of generations for which the coupling holds, with $\tau > k$ meaning that the two explorations are identical, including the out-degrees/marks, up to generation $k$.

The coupled (delayed) marked Galton-Watson process for the DCM has offspring/mark joint distributions
\begin{align}
\mathbb{P}_n\left( \hat N_\emptyset = m, \hat Q_\emptyset = q \right) &= \sum_{s=1}^n 1(D_s^- = m, Q_s = q)  \cdot \frac{1}{n} \qquad \text{and} \label{eq:RootDCMtree} \\
\mathbb{P}_n\left( \hat N_{\bf i} = m, \hat Q_{\bf i} = q, \hat C_{\bf i} = t \right) &= \sum_{s=1}^n 1(D_s^- = m, Q_s = q, \zeta_s/(D_s^+ \vee 1) =t) \cdot \frac{D_s^+}{L_n}, \qquad {\bf i} \neq \emptyset, \label{eq:NodesDCMtree}
\end{align}
for $m \in \mathbb{N}$, $q \in \mathbb{R}$, $t \in [-c, c]$. As for the IRD, we will show that there exists a stopping time $\tau$ such that the event $\{ \tau > k\}$ implies that the exploration of the in-component of a randomly chosen vertex in $\mathcal{G}_n$ and the exploration of the root node of the coupled tree coincides, including out-degrees/marks, up to generation $k$. 

It follows that if the in-components of the randomly chosen vertex and of the root node in the coupled tree are identical up to generation $k$, so are their generalized PageRanks computed up to that level. On the coupled trees, these $k$-level PageRanks are constructed using the notation in Section~\ref{SS.WeightedTree}. Specifically, let 
$$\hat \Pi_\emptyset \equiv 1 \qquad \text{and} \qquad \hat \Pi_{({\bf i},j)} = \hat \Pi_{\bf i}  \hat C_{({\bf i},j)}, \quad {\bf i} \in \mathcal{U}.$$
For any $k \in \mathbb{N}_+$ define the rank at level $k$ of the root node in the tree as:
\begin{equation} \label{eq:PRroot}
\hat R_\emptyset^{(n,k)} = \sum_{l=0}^k \sum_{{\bf j} \in \hat A_l} \hat \Pi_{\bf j} \hat Q_{\bf j},
\end{equation}
where the superscript $n$ refers to the dependence of the branching vectors with the size of the graph, and the superscript $k$ refers to the depth (i.e., number of generations) to which the rank is computed; $\hat A_l$ denotes the set of nodes in the $l$th generation of the tree.  The main argument behind the result in Theorem~\ref{T.PageRank} is a coupling between  $R_\xi$, the rank of a uniformly chosen vertex in $\mathcal{G}_n$, and $\hat R_\emptyset^{(n,k)}$ as defined above. Throughout the paper we use the notation $\mathbb{P}_{n,i}(\cdot) = \mathbb{P}_n( \cdot | \xi = i)$, which implies that the root of the coupled tree has the in-degree and mark corresponding to vertex $i$ in $\mathcal{G}_n$. Equivalently,
\begin{align*}
\mathbb{P}_{n,i} ( \hat N_\emptyset = m, \hat Q_\emptyset = q) &= 1( D_i^- = m, Q_i = q) \hspace{23mm} \text{for the DCM, and} \\
\mathbb{P}_{n,i} ( \hat N_\emptyset = m, \hat Q_\emptyset = q) &= p(m; \Lambda_n^+ \bar W_i^-/(\theta n)) 1( Q_i = q) \qquad \text{for the IRD}.
\end{align*}

The first step in the proof of Theorem~\ref{T.PageRank} is a result that allows us to approximate the PageRank of a randomly chosen vertex, $R_\xi$, with its PageRank computed only using the neighborhood of depth $k$ of each vertex. The exponential rate of convergence in $k$ is due to the damping factor $0 <c <1$, and the result holds for any directed graph. 

\begin{lemma} \label{L.FiniteIterations}
Define 
$${\bf R}^{(n,k)} =  (R_1^{(n,k)}, \dots, R_n^{(n,k)}) = {\bf Q} \sum_{i=0}^k {\bf M}^i,$$
where ${\bf Q} = (Q_1, \dots, Q_n)$, ${\bf M} = \text{\rm diag}({\bf C}) {\bf A}$, ${\bf C} = (C_1, \dots, C_n)$, and ${\bf A}$ is the adjacency matrix of $\mathcal{G}_n$. Then, for any $k \in \mathbb{N}_+$ and $\xi$ uniformly distributed in $\{1, 2, \dots, n\}$, we have
$$\mathbb{E}_n\left[ \left| R_\xi - R_\xi^{(n,k)} \right| \right] \leq \frac{c^{k+1}}{1-c} \cdot \mathbb{E}_n[ |\hat Q_\emptyset|].$$
\end{lemma}

\begin{proof}
We start by writing the scale free generalized PageRank vector ${\bf R}$ as the solution to the system of linear equations
$${\bf R} = {\bf R} {\bf M} + {\bf Q} ,$$
where ${\bf Q} = (Q_1, \dots, Q_n)$, ${\bf M} = \text{diag}({\bf C}) {\bf A}$, ${\bf C} = (C_1, \dots, C_n)$, and ${\bf A}$ is the adjacency matrix of $\mathcal{G}_n$. Note also that
$${\bf R} =   {\bf Q}  \sum_{i=0}^\infty {\bf M}^i.$$
Next, define 
$${\bf R}^{(n,k)} =  {\bf Q}  \sum_{i=0}^k  {\bf M}^i,$$
to be the rank vector computed using only $k$ matrix iterations, and note that Minkowski's inequality gives
\begin{align*}
\| {\bf R} - {\bf R}^{(n,k)} \|_1 &\leq  \sum_{i=k+1}^\infty\norm{ {\bf Q} {\bf M}^i }_1 = \sum_{i=k+1}^\infty\norm{   ({\bf M}^T)^i   {\bf Q}^T}_1 \\
&\leq \sum_{i=k+1}^\infty  \norm{ {\bf M}^T}_1^i \norm{    {\bf Q}^T}_1 \leq \sum_{i=k+1}^\infty  \norm{ {\bf M} }_\infty^i \norm{  {\bf Q}}_1 \\
&\leq \sum_{i=k+1}^\infty c^i \norm{{\bf Q}}_1,
\end{align*}
where we used the observation that $\norm{{\bf M} }_\infty \leq c$. It follows that both for the DCM and the IRG we have
$$\| {\bf R} - {\bf R}^{(n,k)} \|_1 \leq \| {\bf Q} \|_1 (1-c)^{-1} c^{k+1},$$
where $\| {\bf x} \|_1$ is the $L_1$-norm in $\mathbb{R}^n$ and $c \in (0,1)$.  Since $R_\xi$ is a component uniformly chosen at random from the vector ${\bf R}$, we have that
\begin{align*}
 \mathbb{E}_n\left[ \left| R_\xi - R_\xi^{(n,k)} \right| \right] = \mathbb{E}_n\left[ \frac{1}{n}  \| {\bf R} - {\bf R}^{(n,k)} \|_1 \right] \leq \frac{ \| {\bf Q} \|_1}{(1-c) n} \cdot c^{k+1} = \frac{c^{k+1}}{1-c} \cdot \mathbb{E}_n[ |\hat Q_\emptyset|].
\end{align*}
\end{proof}

The next main step will be a coupling between $R_\xi^{(n,k)}$ and $\hat R_\emptyset^{(n,k)}$. Before stating the theorem, we give a preliminary technical lemma that ensures the convergence of $(\hat N_\emptyset, \hat D_\emptyset)$ and $(\hat N_{\bf i}, \hat D_{\bf i})$, ${\bf i} \neq \emptyset$ as $n \to \infty$ in the DCM. We use $d_\text{TV}(F,G)$ to denote the total variation distance between distributions $F$ and $G$. 

\begin{lemma} \label{L.TotalVar}
Define
\begin{align*}
b_n^*(i,j) &= \mathbb{P}_n( \hat N_\emptyset = i, \hat D_{\emptyset} = j) = \frac{1}{n} \sum_{s=1}^n 1(D_s^- = i, D_s^+ = j), \\
b_n(i,j) &= \mathbb{P}_n( \hat N_{\bf i} = i, \hat D_{\bf i} = j)  =  \frac{1}{L_n} \sum_{s=1}^n 1(D_s^- = i, D_s^+ = j) D_s^+, \qquad {\bf i} \neq \emptyset, \\
b^*(i,j) &= P(\mathcal{N}_{\emptyset }= i, \mathcal{D}_\emptyset = j) = P(\mathscr{D}^- = i, \mathscr{D}^+ = j), \\
b(i,j) &= P(\mathcal{N}_{\bf i} = i, \mathcal{D}_{\bf i} = j)  =  E\left[ 1(\mathscr{D}^- = i, \mathscr{D}^+ = j) \frac{\mathscr{D}^+}{E[\mathscr{D}^+]} \right], \qquad {\bf i} \neq \emptyset,
\end{align*}
and let $B_n^*, B^*, B_n, B$ denote its corresponding distribution functions. Then, under Assumption~\ref{A.Primitives}(A), we have that for any $x_n \geq 1$, 
\begin{align*}
d_\text{TV}(B_n^*, B^*) &\leq d_1(F_n, F), \\
d_\text{TV}(B_n, B) &\leq \left( 2x_n +1 + \frac{1}{E[\mathscr{D}^+]} \right) d_1(F_n, F) + 2 E[\mathscr{D}^+ 1(\mathscr{D}^+ > x_n)]. 
\end{align*}
\end{lemma}

\begin{proof}
Note that by Markov's inequality we have
\begin{align*}
d_\text{TV}(B_n^*, B^*) &= \inf_\pi \mathbb{P}_n \left( (\hat N_\emptyset, \hat D_\emptyset) \neq (\mathscr{D}^-, \mathscr{D}^+) \right) \leq \inf_\pi \mathbb{E}_n\left[ |\hat N_\emptyset - \mathscr{D}^- | + | \hat D_\emptyset - \mathscr{D}^+ | \right] \\
&\leq d_1(B_n^*, B^*) \leq d_1(F_n, F),
\end{align*}
where the infimum is taken over all distributions $\pi$ with marginals $B_n^*$ and $B^*$ and $d_1$ denotes the Kantorovich-Rubinstein distance. Bounding the total variation distance between $B_n$ and $B$ requires more work, since under our current assumptions $d_1(B_n, B)$ may be infinite. To start, let $(\hat N_\emptyset, \hat D_\emptyset, \mathscr{D}^-, \mathscr{D}^+)$ be an optimal coupling for which $d_1(B_n^*, B^*)$ is attained, and note that:
\begin{align*}
d_\text{TV}(B_n, B) &= \sum_{i=0}^n \sum_{j=1}^\infty  \left|  b_n(i,j) - b(i,j)  \right| \\
&= \sum_{i=0}^n \sum_{j=1}^\infty \left| \mathbb{E}_n \left[ 1(\hat N_\emptyset = i, \hat D_\emptyset = j) \frac{n \hat D_\emptyset}{L_n} \right] - E\left[ 1(\mathscr{D}^- =i, \mathscr{D}^+ = j) \frac{\mathscr{D}^+}{E[\mathscr{D}^+]} \right] \right| \\
&\leq \sum_{i=0}^n \sum_{j=1}^\infty \mathbb{E}_n \left[ 1(\hat N_\emptyset = i, \hat D_\emptyset = j) \hat D_\emptyset \right] \left| \frac{n}{L_n} - \frac{1}{E[\mathscr{D}^+]}   \right| \\
&\hspace{5mm} + \frac{1}{E[\mathscr{D}^+]}  \sum_{i=0}^n \sum_{j=1}^\infty \left| \mathbb{E}_n \left[ 1(\hat N_\emptyset = i, \hat D_\emptyset = j) \hat D_\emptyset \right] - E\left[ 1(\mathscr{D}^- =i, \mathscr{D}^+ = j) \mathscr{D}^+ \right] \right| \\
&= \frac{L_n}{n} \left| \frac{n}{L_n} - \frac{1}{E[\mathscr{D}^+]}   \right| + \sum_{i=0}^n \sum_{j=1}^\infty j \left| b_n^*(i,j) - b^*(i,j)  \right|.
\end{align*}
Now note that
$$ \frac{L_n}{n} \left| \frac{n}{L_n} - \frac{1}{E[\mathscr{D}^+]}   \right| = \frac{1}{E[\mathscr{D}^+]} \left| E[\mathscr{D}^+] - \mathbb{E}_n[ \hat D_\emptyset] \right| \leq \frac{1}{E[\mathscr{D}^+]} \cdot d_1(F_n, F),$$
and for any $x_n \geq 1$ we have
\begin{align*}
\sum_{i=0}^n \sum_{j=1}^\infty j \left| b_n^*(i,j) - b^*(i,j)  \right| &\leq \sum_{i=0}^\infty \sum_{j=1}^{\lfloor x_n \rfloor} x_n \left|b_n^*(i,j) - b^*(i,j) \right| + \sum_{i=0}^\infty \sum_{j=\lfloor x_n \rfloor +1} j (b_n^*(i,j) + b^*(i,j)) \\
&\leq x_n d_\text{TV}(B_n^*, B^*) + \mathbb{E}_n[ \hat D_\emptyset 1( \hat D_\emptyset > x_n) ] + E[ \mathscr{D}^+ 1(\mathscr{D}^+ > x_n) ] \\
&\leq (2x_n+1) d_1(F_n, F) + 2 E[ \mathscr{D}^+ 1(\mathscr{D}^+ > x_n)],
\end{align*}
where in the last step we used the observation that
\begin{align*}
\mathbb{E}_n[ \hat D_\emptyset 1( \hat D_\emptyset > x_n) ] &= \mathbb{E}_n[ (\hat D_\emptyset - x_n)^+] + x_n \mathbb{P}_n(\hat D_\emptyset > x_n) \\
&\leq d_1(B_n^*, B) + E[ (\mathscr{D}^+ - x_n)^+] + x_n d_\text{TV}(B_n^*, B^*) + x_n P(\mathscr{D}^+ > x_n) \\
&\leq (x_n+1) d_1(F_n, F) + E[ \mathscr{D}^+ 1(\mathscr{D}^+ > x_n)]. 
\end{align*}
It follows that 
$$d_\text{TV}(B_n, B) \leq \left( 2 x_n + 1 + \frac{1}{E[\mathscr{D}^+]} \right) d_1(F_n, F) + 2 E[ \mathscr{D}^+ 1(\mathscr{D}^+ > x_n)].$$
\end{proof}

\bigskip

We are now ready to state our coupling result for $R_\xi^{(n,k)}$ and $\hat R_\emptyset^{(n,k)}$. 

\begin{theo} \label{T.CouplingTime}
Under Assumption~\ref{A.Primitives}, there exists a coupling $(R_\xi^{(n,k)}, \hat R_\emptyset^{(n,k)})$, where $R_\xi^{(n,k)}$ is defined as in Lemma~\ref{L.FiniteIterations} and  $\hat R_\emptyset^{(n,k)}$ is constructed according to \eqref{eq:PRroot}, and a stopping time $\tau$, such that for any fixed $k \in\mathbb{N}_+$,
$$\mathbb{P}_n\left( R_\xi^{(n,k)} \neq \hat R_\emptyset^{(n,k)} \right) \leq \frac{1}{n} \sum_{i=1}^n \mathbb{P}_{n,i}( \tau \leq k) \xrightarrow{P} 0, \qquad n \to \infty.$$
\end{theo}

\begin{proof}
We consider the two models separately.

{\bf Inhomogeneous random digraph:}

For the IRD, the statement of the theorem is that of Theorem~3.7 in \cite{Lee_Olv_19}, provided we redefine the stopping time so that $\{ \tau > k\}$ in this paper corresponds to $\{ \tau > 2k\}$ in \cite{Lee_Olv_19}, since the coupling described there consists of odd and even steps and is such that generation $k$ is completed after $2k$ steps of the exploration process. 

{\bf Directed configuration model:}

The proof for the DCM requires that we modify the proof of Lemma~5.4 in \cite{Chen_Lit_Olv_17} to avoid the stronger moment conditions assumed there. The exploration of the graph and the simultaneous construction of the marked tree is the same as in \cite{Chen_Lit_Olv_17}, that is, we perform a breadth-first exploration of the in-component of the randomly chosen vertex by selecting uniformly at random from all $L_n$ outbound half-edges, rejecting those that have been selected earlier; we construct the coupled tree by using the in-degree of the vertex whose outbound half-edge has been selected as the number of offspring for the node being explored, and record its out-degree, weight and personalization value as its mark. The coupling breaks the first time we draw an outbound half-edge belonging to a vertex that has already been been explored. It follows that the probability that the coupling breaks while pairing a vertex at distance $r$ from the randomly chosen vertex, is bounded from above by:
$$P_r := \frac{1}{L_n} \sum_{j=0}^r \hat V_j, \qquad r \geq 0,$$
where $\hat V_j =  \sum_{{\bf i} \in \hat A_j} \hat D^+_{\bf i}$ is the sum of the out-degrees (which we include as part of the marks) of the nodes in the $j$th generation of the marked tree (denoted $\hat A_j$). Let $\hat Z_r = |\hat A_r|$ denote the number of individuals in the $r$th generation of the tree and define $|\hat T_k| = \sum_{r=0}^k \hat Z_r$, $|\hat I_k| = \sum_{r=0}^k \hat V_r$. Next, let $a_n \geq 1$ be a sequence to be determined later, and note that 
\begin{align*}
\mathbb{P}_{n,i}(\tau \leq k) &\leq \mathbb{P}_{n,i}(\tau \leq k, |\hat T_k| \vee |\hat I_k| \leq a_n) +  \mathbb{P}_{n,i} \left( |\hat T_k| \vee |\hat I_k| > a_n \right) \\
&\leq \sum_{r=1}^k \mathbb{P}_{n,i}\left(\tau = r, |\hat T_k| \vee |\hat I_k| \leq a_n \right) + \mathbb{P}_{n,i} \left( |\hat T_k| \vee |\hat I_k| > a_n \right) \\
 & \leq \sum_{r=1}^k \mathbb{P}_{n,i} \left( \text{Bin}(\hat Z_{r}, P_r) \geq 1,  |\hat T_r| \vee |\hat I_r| \leq a_n \right) + \mathbb{P}_{n,i} \left( |\hat T_k| \vee |\hat I_k| > a_n \right) \\
 &\leq \sum_{r=1}^k \mathbb{P}_n\left( \text{Bin}( a_n, a_n/L_n) \geq 1 \right) + \mathbb{P}_{n,i} \left( |\hat T_k| \vee |\hat I_k| > a_n \right) \\
  &\leq \frac{k a_n^2}{L_n} + \mathbb{P}_{n,i} \left( |\hat T_k| \vee |\hat I_k| > a_n \right),
\end{align*}
where $\text{Bin}(n,p)$ is a Binomial random variable with parameters $(n,p)$. To analyze the last probability we use Lemma~4.6 in \cite{Lee_Olv_19} to couple the marked Galton-Watson process constructed using $\{ (\hat N_{\bf i}, \hat D_{\bf i}): {\bf i} \in \mathcal{U} \}$ with another marked Galton-Watson process constructed using $\{ (\mathcal{N}_{\bf i}, \mathcal{D}_{\bf i}): {\bf i} \in \mathcal{U}\}$ (the latter does not depend on $\mathscr{F}_n$). In particular, if we define $B_n^*, B^*, B_n, B$ as in Lemma~\ref{L.TotalVar}, then Lemma~4.6 in \cite{Lee_Olv_19} (see the last line of the proof) gives
\begin{align*}
\frac{1}{n} \sum_{i=1}^n \mathbb{P}_{n,i} \left( |\hat T_k| \vee |\hat I_k| > a_n \right) &\leq P\left( |\mathcal{T}_k| \vee |\mathcal{I}_k| > a_n \right) + d_\text{TV}(B_n^*, B^*) + k a_n d_\text{TV}(B_n, B) ,
\end{align*}
where $|\mathcal{T}_k| := \sum_{r=0}^k |A_r|$ and $|\mathcal{I}_k| := \sum_{r=0}^k \sum_{{\bf i} \in A_r} \mathcal{D}_{\bf i}$, $A_r$ is the $r$th generation of the coupled tree, and $d_\text{TV}(F,G)$ is the total variation distance between distributions $F$ and $G$. Moreover, by Lemma~\ref{L.TotalVar} we have 
\begin{align*}
d_\text{TV}(B_n^*, B^*) &\leq d_1(F_n, F), \quad \text{and} \\
d_\text{TV}(B_n, B) &\leq \left( 2 x_n + 1 + \frac{1}{E[\mathscr{D}^+]} \right) d_1(F_n, F) + 2 E[ \mathscr{D}^+ 1(\mathscr{D}^+ > x_n)],
\end{align*}
for any $x_n \geq 1$. These in turn yield
\begin{align*}
\frac{1}{n} \sum_{i=1}^n \mathbb{P}_{n,i}(\tau \leq k) &\leq \frac{ka_n^2}{L_n} + P(|\mathcal{T}_k| \vee |\mathcal{I}_k| > a_n) + d_1 (F_n, F) \\
&\hspace{5mm} + k a_n \left\{ \left( 2 x_n + 1 + \frac{1}{E[\mathscr{D}^+]} \right) d_1(F_n, F) + 2 E[ \mathscr{D}^+ 1(\mathscr{D}^+ > x_n)] \right\}.
\end{align*}
Finally, pick $x_n = d_1(F_n, F)^{-1/4}$ and $a_n = \min\{ d_1(F_n, F)^{-1/4}, E[\mathscr{D}^+ 1(\mathscr{D}^+ > x_n)]^{-1/2}, n^{1/4}\}$, and use Assumption~\ref{A.Primitives} to obtain that $a_n, x_n \xrightarrow{P} \infty$. Since $|\mathcal{T}_k| \vee |\mathcal{I}_k| < \infty$ a.s., we obtain 
$$\frac{1}{n} \sum_{i=1}^n \mathbb{P}_{n,i}(\tau \leq k) \xrightarrow{P} 0, \qquad n \to \infty.$$
This completes the proof. 
\end{proof}

Before showing how Theorem~\ref{T.CouplingTime} can be used to obtain a coupling between $R_\xi^{(n,k)}$ and $\hat R_\emptyset^{(n,k)}$ in $L_1$ norm, we first show that $\hat R_\emptyset^{(n,k)}$ converges in $d_1$ as $n \to \infty$.  The following lemma is the key to establish this convergence, which will then follow from Theorem~3 in \cite{chenolvera1}. Although similar results appear in \cite{Chen_Lit_Olv_17} and \cite{Lee_Olv_19}, they depend on the convergence in $d_1$ of the vector $(\hat N_1, \hat Q_1, \hat C_1)$ to $(\mathcal{N}, \mathcal{Q}, \mathcal{C})$, which is not guaranteed under our current assumptions; in fact, it is possible to have $E[\mathcal{N} + |\mathcal{Q}|] = \infty$, which would imply that $d_1$ is not even well defined. 

\begin{lemma} \label{L.KRconvergence}
Let $\nu_n$ denote the probability measure of the vector
$$ ( \hat C_1 \hat Q_1,  \hat C_1 1(\hat N_1 \geq 1),  \hat C_1 1(\hat N_1 \geq 2), \dots),$$
and let $\nu$ denote the probability measure of the vector
$$( \mathcal{CQ}, \mathcal{C} 1(\mathcal{N} \geq 1), \mathcal{C} 1(\mathcal{N} \geq 2), \dots).$$
Then, under Assumption~\ref{A.Primitives}, we have that, for both the DCM and the IRD, 
$$d_1(\nu_n, \nu) \stackrel{P}{\to} 0, \qquad n \to \infty.$$
\end{lemma}

\begin{proof}
As before, we consider the two models separately. 

{\bf Directed configuration model:}

Let $(D_{(n)}^-, D_{(n)}^+, Q_{(n)}, \zeta_{(n)})$ be a vector distributed according to $F_n$ and let $(\mathscr{D}^-, \mathscr{D}^+, Q, \zeta)$ be distributed according to $F$. We will first show that
$$ ( \hat C_1 \hat Q_1,  \hat C_1 1(\hat N_1 \geq 1),  \hat C_1 1(\hat N_1 \geq 2), \dots) \Rightarrow ( \mathcal{CQ}, \mathcal{C} 1(\mathcal{N} \geq 1), \mathcal{C} 1(\mathcal{N} \geq 2), \dots),$$
as $n \to \infty$, where $\Rightarrow$ denotes weak convergence. To this end, let $h: \mathbb{R}^\infty \to \mathbb{R}$ be a bounded and continuous function on $\mathbb{R}^\infty$ (equipped with the metric $\| {\bf x} \|_1 = \sum_{i =1}^\infty |x_i|$), and note that for any $M > 0$ we have
\begin{align*}
&\mathbb{E}_n\left[ h(\hat C_1 \hat Q_1, \hat C_1 1(\hat N_1 \geq 1), \hat C_1 1(\hat N_1 \geq 2), \dots) \right] \\
&= \frac{n}{L_n} \mathbb{E}_n \left[ D^+_{(n)} h\left( \frac{\zeta_{(n)} Q_{(n)}}{D_{(n)}^+ \vee 1}, \frac{\zeta_{(n)} 1(D^-_{(n)} \geq 1)}{D_{(n)}^+ \vee 1}, \frac{\zeta_{(n)} 1(D^-_{(n)} \geq 2)}{D_{(n)}^+ \vee 1}, \dots \right) \right] \\
&\leq \frac{n}{L_n} \mathbb{E}_n \left[ (D^+_{(n)} \wedge M) h\left( \frac{\zeta_{(n)} Q_{(n)}}{D_{(n)}^+ \vee 1}, \frac{\zeta_{(n)} 1(D^-_{(n)} \geq 1)}{D_{(n)}^+ \vee 1}, \frac{\zeta_{(n)} 1(D^-_{(n)} \geq 2)}{D_{(n)}^+ \vee 1}, \dots \right) \right] \\
&\hspace{5mm} + \frac{n}{L_n}  \mathbb{E}_n \left[ (D^+_{(n)} - M)^+ \right] \sup_{{\bf x} \in \mathbb{R}^\infty} |h({\bf x})|.
\end{align*}
Since $J_M(m,d,q,c) := (d \wedge M) h(cq/(d \vee 1), c 1(m\geq 1)/(d \vee 1), c 1(m \geq 2)/(d \vee 1), \dots)$ is bounded and continuous on $\mathbb{N} \times \mathbb{N} \times \mathbb{R} \times [-c,c]$, and $(d-M)^+$ is Lipchitz continuous, Assumption~\ref{A.Primitives} yields
\begin{align*}
&\limsup_{n \to \infty} \mathbb{E}_n\left[ h(\hat C_1 \hat Q_1, \hat C_1 1(\hat N_1 \geq 1), \hat C_1 1(\hat N_1 \geq 2), \dots) \right] \\
& \leq \frac{1}{E[\mathscr{D}^+]} E[ J_M(\mathscr{D}^-, \mathscr{D}^+, Q, \zeta)] + \frac{1}{E[\mathscr{D}^+]} E[ (\mathscr{D}^+-M)^+]  \sup_{{\bf x} \in \mathbb{R}^\infty} |h({\bf x})|,
\end{align*}
with the limit holding in probability. The same arguments also yield
\begin{align*}
&\liminf_{n \to \infty} \mathbb{E}_n\left[ h(\hat C_1 \hat Q_1, \hat C_1 1(\hat N_1 \geq 1), \hat C_1 1(\hat N_1 \geq 2), \dots) \right] \\
& \geq \frac{1}{E[\mathscr{D}^+]} E[ J_M(\mathscr{D}^-, \mathscr{D}^+, Q, \zeta)] - \frac{1}{E[\mathscr{D}^+]} E[ (\mathscr{D}^+-M)^+]  \sup_{{\bf x} \in \mathbb{R}^\infty} |h({\bf x})|,
\end{align*}
in probability. Now take $M \to \infty$ and use the monotone convergence theorem to obtain that
\begin{align*}
&\mathbb{E}_n\left[ h(\hat C_1 \hat Q_1, \hat C_1 1(\hat N_1 \geq 1), \hat C_1 1(\hat N_1 \geq 2), \dots) \right] \\
&\xrightarrow{P} \frac{1}{E[\mathscr{D}^+]} E\left[ \lim_{M \to \infty} J_M(\mathscr{D}^-, \mathscr{D}^+, Q, \zeta) \right] \\
&=  \frac{1}{E[\mathscr{D}^+]} E\left[ \mathscr{D}^+ h\left( \frac{\zeta Q}{\mathscr{D}^+ \vee 1} , \frac{\zeta 1(\mathscr{D}^- \geq 1)}{\mathscr{D}^+ \vee 1}, \frac{\zeta 1(\mathscr{D}^- \geq 2)}{\mathscr{D}^+ \vee 1}, \dots \right) \right] \\
&= E[ h(\mathcal{C Q}, \mathcal{C} 1(\mathcal{N} \geq 1), \mathcal{C} 1(\mathcal{N} \geq 2), \dots ) ] 
\end{align*}
as $n \to \infty$. 

To establish the convergence in $d_1$ it suffices to show that 
$$\mathbb{E}_n\left[ |\hat Q_1 \hat C_1| + |\hat C_1| 1(\hat N_1 \geq 1) + |\hat C_1| 1(\hat N_1 \geq 2) + \dots \right] \xrightarrow{P} E\left[ |\mathcal{QC}| + |\mathcal{C}| 1(\mathcal{N} \geq 1) + |\mathcal{C}| 1(\mathcal{N} \geq 2) + \dots \right]$$
as $n \to \infty$ (see Theorem~6.9 and Definition~6.8(i) in \cite{Villani_2009}). To see that this is indeed the case, let $(D_{(n)}^-, D_{(n)}^+, Q_{(n)}, \zeta_{(n)}, \mathscr{D}^-, \mathscr{D}^+, Q, \zeta)$ be an optimal coupling for which $d_1(F_n, F)$ is attained. Next, note that
\begin{align*}
\left| \mathbb{E}_n\left[ |\hat Q_1 \hat C_1| \right] - E[ |\mathcal{Q C}| ] \right| &=  \left| \frac{n}{L_n} \mathbb{E}_n\left[   \frac{ D_{(n)}^+ |Q_{(n)} \zeta_{(n)}|}{D_{(n)}^+ \vee 1} \right] - \frac{1}{E[\mathscr{D}^+]} E\left[ \frac{\mathscr{D}^+ |Q\zeta|}{\mathscr{D}^+ \vee 1} \right] \right| \\
&\leq \frac{n}{L_n} \left|  \mathbb{E}_n\left[   \frac{ D_{(n)}^+ |Q_{(n)} \zeta_{(n)}|}{D_{(n)}^+ \vee 1} - \frac{\mathscr{D}^+ |Q\zeta|}{\mathscr{D}^+ \vee 1}  \right] \right| + \left| \frac{n}{L_n} - \frac{1}{E[\mathscr{D}^+]} \right| E\left[ \frac{\mathscr{D}^+ |Q\zeta|}{\mathscr{D}^+ \vee 1} \right]  \\
&\leq  \frac{n}{L_n}   \mathbb{E}_n\left[   \left| |Q_{(n)} \zeta_{(n)}| - |Q \zeta| \right| \right]   +  \frac{n}{L_n} \left| \mathbb{E}_n\left[  \frac{ D_{(n)}^+ |Q \zeta|}{D_{(n)}^+ \vee 1}  - \frac{\mathscr{D}^+ |Q\zeta|}{\mathscr{D}^+ \vee 1}  \right] \right| \\
&\hspace{5mm} + \left| \frac{n}{L_n} - \frac{1}{E[\mathscr{D}^+]} \right| E\left[  |Q c| \right] \\
&\leq \frac{n}{L_n}  \mathbb{E}_n\left[ |Q_{(n)} - Q| c + |Q| |\zeta_{(n)} - \zeta| \right] + \frac{n}{L_n} \mathbb{E}_n\left[ c |Q| \left|   \frac{D_{(n)}^+}{D_{(n)}^+ \vee 1} - \frac{\mathscr{D}^+}{\mathscr{D}^+ \vee 1} \right| \right] \\
&\hspace{5mm} + \left| \frac{n}{L_n} - \frac{1}{E[\mathscr{D}^+]} \right|  E\left[ c |Q | \right] .
\end{align*}
It follows from Assumption~\ref{A.Primitives} and the dominated convergence theorem that
$$\left| \mathbb{E}_n\left[ |\hat Q_1 \hat C_1| \right] - E[ |\mathcal{Q C}| ] \right| \xrightarrow{P} 0$$
as $n \to \infty$.  A similar argument gives for any $i \geq 1$, 
\begin{align*}
&\sum_{i=1}^\infty \left| \mathbb{E}_n\left[  |\hat C_1| 1(\hat N_1 \geq i) \right] - E[ |\mathcal{C}| 1(\mathcal{N} \geq i) ] \right| \\
&\leq \sum_{i=1}^\infty \left( \frac{n}{L_n}  \mathbb{E}_n\left[ |1(D_{(n)}^- \geq i) -  1(\mathscr{D}^- \geq i) | c + 1(\mathscr{D}^- \geq i) |\zeta_{(n)} - \zeta| \right] \right. \\
&\hspace{5mm} \left. + \frac{n}{L_n} \mathbb{E}_n\left[ c 1(\mathscr{D}^- \geq i) \left|   \frac{D_{(n)}^+}{D_{(n)}^+ \vee 1} - \frac{\mathscr{D}^+}{\mathscr{D}^+ \vee 1} \right| \right] + \left| \frac{n}{L_n} - \frac{1}{E[\mathscr{D}^+]} \right|  E\left[ c 1(\mathscr{D}^- \geq i) \right]  \right) \\
&\leq \frac{n}{L_n} \mathbb{E}_n\left[ |D_{(n)}^- - \mathscr{D}^-|c + \mathscr{D}^- |\zeta_{(n)} - \zeta|  + c \mathscr{D}^- \left|   \frac{D_{(n)}^+}{D_{(n)}^+ \vee 1} - \frac{\mathscr{D}^+}{\mathscr{D}^+ \vee 1} \right| \right]  \\
&\hspace{5mm} +  \left| \frac{n}{L_n} - \frac{1}{E[\mathscr{D}^+]} \right|  E[c \mathscr{D}^-] \xrightarrow{P} 0
\end{align*}
as $n \to \infty$. This completes the proof for the DCM.

\bigskip
{\bf Inhomogeneous random digraph:}

Let $(Z_{(n)}^-, Z_{(n)}^+ , Q_{(n)}, \zeta_{(n)})$ and $(Z^-, Z^+, Q, \zeta)$ be distributed according to:
\begin{align*}
\tilde f_n(m,d,q,z) &:= \mathbb{P}_n\left( Z_{(n)}^- = m, Z_{(n)}^+ = d, Q_{(n)} = q, \zeta_{(n)} = z \right) \\
&= \frac{1}{n} \sum_{i=1}^n p(m; \Lambda_n^+ \bar W_i^-/(\theta n))  p(d; \Lambda_n^- \bar W_i^+/(\theta n)) 1(Q_i = q, \zeta_i = z),  \quad \text{and} \\
\tilde f(m,d,q,z) &:= P\left( Z^- = m, Z^+ = d, Q \in dq, \zeta \in dz \right) \\
&= E\left[ p(m; E[W^+] W^-/\theta) p(d; E[W^-] W^+/\theta) 1(Q \in dq, \zeta \in dz) \right],
\end{align*}
respectively, where $p(m;\lambda) = e^{-\lambda} \lambda^m/m!$. Let $\tilde F_n$ and $\tilde F$ denote their corresponding distribution functions. The proof for the IRD will follow essentially from the same arguments used for the DCM once we show that $d_1(\tilde F_n, \tilde F) \to 0$ as $n \to \infty$ (simply replace $(D_{(n)}^-, D_{(n)}^+, \mathscr{D}^-, \mathscr{D}^+)$ with $(Z_{(n)}^-, Z_{(n)}^+, Z^-, Z^+)$,  $D_{(n)}^+ \vee 1$ with $Z_{(n)}^+ +1$, and $\mathscr{D}^+ \vee 1$ with $Z^+ + 1$). To show the convergence in $d_1$ note that if $(W_{(n)}^- , W_{(n)}^+, Q_{(n)}, \zeta_{(n)})$ is distributed according to $F_n$, then
$$(R^{-1}(U^-; \Lambda_n^+ \bar W^-_{(n)}/(\theta n)), R^{-1}(U^+; \Lambda_n^- \bar W^+_{(n)}/(\theta n)), Q_{(n)}, \zeta_{(n)}) \stackrel{\mathcal{D}}{=} (Z_{(n)}^-, Z_{(n)}^+ , Q_{(n)}, \zeta_{(n)}),$$
where $R(x;\lambda) = \sum_{m=0}^{\lfloor x \rfloor} p(m;\lambda)$ is the distribution function of a Poisson random variable with mean $\lambda$, $R^{-1}(u;\lambda) = \inf\{ x: R(x;\lambda) \geq u\}$ is its generalized inverse, and $U^-, U^+$ are i.i.d.~Uniform$[0,1]$ random variables independent of $(W_{(n)}^- , W_{(n)}^+, Q_{(n)}, \zeta_{(n)})$. Since $R(x;\lambda)$ is continuous in $\lambda$, the continuous mapping theorem yields
\begin{align*}
&(R^{-1}(U^-; \Lambda_n^+ \bar W^-_{(n)}/(\theta n)), R^{-1}(U^+; \Lambda_n^- \bar W^+_{(n)}/(\theta n)), Q_{(n)}, \zeta_{(n)}) \\
&\Rightarrow (R^{-1}(U^-; E[W^+] W^-/\theta), R^{-1}(U^+; E[W^-] W^+/\theta), Q, \zeta) \\
&\stackrel{\mathcal{D}}{=} (Z^-, Z^+, Q, \zeta)
\end{align*}
as $n \to \infty$. The convergence of the first absolute moment follows from noting that
\begin{align*}
\mathbb{E}_n\left[ \| (Z_{(n)}^-, Z_{(n)}^+ , Q_{(n)}, \zeta_{(n)}) \|_1 \right] &= \mathbb{E}_n\left[ \frac{\Lambda_n^+ W_{(n)}^-}{\theta n} + \frac{\Lambda_n^- W_{(n)}^+}{\theta n} + | Q_{(n)}| + |\zeta_{(n)}| \right] \\
&\xrightarrow{P} E\left[ \frac{E[W^+] W^-}{\theta} + \frac{E[W^-] W^+}{\theta} + |Q| + |\zeta| \right] \\
&= E\left[ \| (Z^-, Z^+, Q, \zeta) \|_1 \right]
\end{align*}
as $n \to \infty$. We conclude that $d_1(\tilde F_n, \tilde F) \xrightarrow{P} 0$ as $n \to \infty$, which completes the proof. 
\end{proof}

\bigskip

We now use Lemma~\ref{L.KRconvergence} to obtain the convergence of $\hat R_\emptyset^{(n,k)}$ in $d_1$ as $n \to \infty$.

\begin{theo} \label{T.TreeConvergence}
Under Assumption~\ref{A.Primitives}, there exists a coupling $(\hat R_\emptyset^{(n,k)}, \mathcal{R}^{(k)})$, where
$$\mathcal{R}^{(k)} := \sum_{r=0}^k \sum_{{\bf j} \in A_r} \Pi_{\bf j} \mathcal{Q}_{\bf j}$$
with $(\mathcal{N}_\emptyset, \mathcal{Q}_\emptyset) \stackrel{\mathcal{D}}{=} (\mathcal{N}_0, \mathcal{Q}_0)$ and $\{ (\mathcal{N}_{\bf i}, \mathcal{Q}_{\bf i}, \mathcal{C}_{\bf i}) \}_{{\bf i} \neq \emptyset}$ i.i.d.~with the same distribution as $(\mathcal{N}, \mathcal{Q}, \mathcal{C})$ and independent of $(\mathcal{N}_\emptyset, \mathcal{Q}_\emptyset)$ (as defined in Theorem~\ref{T.PageRank}), such that for any fixed $k \in \mathbb{N}_+$, 
$$\mathbb{E}_n\left[ \left| \hat R_\emptyset^{(n,k)} - \mathcal{R}^{(k)} \right| \right]  \xrightarrow{P} 0, \qquad n \to \infty.$$
\end{theo}

\begin{proof}
We start by defining
$$\hat X_j^{(n,k-1)} = \sum_{r=1}^{k} \sum_{(j,{\bf i}) \in \hat A_r} \hat \Pi_{(j,{\bf i})}   \hat Q_{(j,{\bf i})} ,$$
and noting that
\begin{align*}
\hat R_\emptyset^{(n,k)} &=  \hat Q_\emptyset +\sum_{j=1}^{\hat N_\emptyset} \hat X_j^{(n,k-1)},
\end{align*}
with the $\{ \hat X_j^{(n,k-1)}\}$ conditionally i.i.d.~and conditionally independent of $(\hat Q_\emptyset, \hat N_\emptyset)$, given $\mathscr{F}_n$. Moreover, as described in Section~\ref{SS.WeightedTree}, the $\{ X_j^{(n,k-1)} \}_{j \geq 1}$ satisfy
$$\hat X_1^{(n,k-1)} \stackrel{\mathcal{D}}{=} \hat C_1 \hat Q_1 + \sum_{i=1}^{\hat N_1}  \hat C_1 \hat X_i^{(n,k-2)},$$
with the $\{ X_i^{(n,k-2)}\}_{i \geq 1}$ conditionally i.i.d.~and conditionally independent of $(\hat Q_1, \hat N_1, \hat C_1)$, given $\mathscr{F}_n$.  Let $\nu_n$ and $\nu$ denote the probability measures (conditionally on $\mathscr{F}_n$) of the vectors
\begin{align*}
( \hat C_1 \hat Q_1,  \hat C_1 1(\hat N_1 \geq 1), \hat C_1 1(\hat N_1 \geq 2), \dots) \quad \text{and} \quad ( \mathcal{CQ}, \mathcal{C} 1(\mathcal{N} \geq 1), \mathcal{C} 1(\mathcal{N} \geq 2), \dots ),
 \end{align*}
respectively. Since by Lemma~\ref{L.KRconvergence} we have that
$$d_1(\nu_n, \nu) \stackrel{P}{\longrightarrow} 0, \qquad n \to \infty,$$
it follows by Theorem~2 (Case 1) in \cite{Chen_Olv_13} (applied conditionally given $\mathscr{F} = \sigma\left( \bigcup_{n=1}^\infty \mathscr{F}_n \right)$), that 
$$d_1(G_{k-1,n}, G_{k-1}) \xrightarrow{P} 0, \qquad n \to \infty,$$
where $G_{k-1,n}(x) = \mathbb{P}_n \left( \hat X_1^{(n,k-1)} \leq x \right)$ and $G_{k-1}(x) =P(X_1^{(k-1)} \leq x)$, with 
$$X_j^{(k-1)} = \sum_{r=1}^k \sum_{(j,{\bf i}) \in A_r} \Pi_{(j,{\bf i})} \mathcal{Q}_{(j,{\bf i})}.$$
Next, note that Assumption~\ref{A.Primitives} implies that $(\hat N_\emptyset, \hat Q_\emptyset)$ converges to $(\mathcal{N}_0, \mathcal{Q}_0)$ in $d_1$, so we can choose an optimal coupling $(\hat N_\emptyset, \hat Q_\emptyset, \mathcal{N}_0, \mathcal{Q}_0)$. Let $\left\{ (\hat X_j^{(n,k-1)}, X_j^{(k-1)}): j \geq 1\right\}$ be a sequence of i.i.d.~vectors sampled according to an optimal coupling for $d_1(G_{k-1,n}, G_{k-1})$, conditionally independent  (given $\mathscr{F}_n$) of $(\hat N_\emptyset, \hat Q_\emptyset, \mathcal{N}_0, \mathcal{Q}_0)$, and construct
$$\hat R_\emptyset^{(n,k)} =  \hat Q_\emptyset + \sum_{j=1}^{\hat N_\emptyset} \hat X_j^{(n,k-1)} \qquad \text{and} \qquad  \mathcal{R}^{(k)} =  \mathcal{Q}_0 + \sum_{j=1}^{\mathcal{N}_0} X_j^{(k-1)} .$$
We then have,
\begin{align*}
\mathbb{E}_n\left[ \left| \hat R_\emptyset^{(n,k)} - \mathcal{R}^{(k)} \right| \right] &\leq  \mathbb{E}_n\left[ \left| \hat Q_\emptyset - \mathcal{Q}_0 \right| + \sum_{j=1}^{\hat N_\emptyset \wedge \mathcal{N}_0} \left|  \hat X_j^{(n,k-1)}  - X_j^{(k-1)} \right| \right. \\
&\hspace{5mm} \left.  + \sum_{j=\hat N_\emptyset \wedge \mathcal{N}_0 +1}^{\hat N_\emptyset} |\hat X_j^{(n,k-1)}|  + \sum_{j=\hat N_\emptyset \wedge \mathcal{N}_0 +1}^{\mathcal{N}_0} |X_j^{(k-1)}| \right]  \\
&=  \mathbb{E}_n\left[ \left| \hat Q_\emptyset - \mathcal{Q}_0 \right|  \right] + \mathbb{E}_n[ \hat N_\emptyset \wedge \mathcal{N}_0] d_1(G_{k-1,n}, G_{k-1}) \\
&\hspace{5mm} + \mathbb{E}_n\left[ (\hat N_\emptyset - \mathcal{N}_0)^+ \right] \mathbb{E}_n\left[ |\hat X_j^{(n,k-1)}| \right] + \mathbb{E}_n\left[ (\mathcal{N}_0 - \hat N_\emptyset)^+ \right] E\left[ |X_j^{(k-1)}| \right] \\
&\leq \mathbb{E}_n\left[ \left| \hat Q_\emptyset - \mathcal{Q}_0 \right|  \right] + E[  \mathcal{N}_0] d_1(G_{k-1,n}, G_{k-1}) \\
&\hspace{5mm} + \mathbb{E}_n\left[ \left|\hat N_\emptyset - \mathcal{N}_0 \right| \right] \left( E\left[ |X_j^{(k-1)}| \right]  + d_1(G_{k-1,n}, G_{k-1})  \right) \xrightarrow{P} 0,
\end{align*}
as $n \to \infty$. This completes the proof. 
\end{proof}

\bigskip

We now use Theorem~\ref{T.CouplingTime} to show that there exists a coupling between $R_\xi^{(n,k)}$ and $\hat R_\emptyset^{(n,k)}$ such that their difference converges to zero in $L_1$ norm. 

\begin{theo} \label{T.Coupling}
Under Assumption~\ref{A.Primitives}, there exists a coupling $(R_\xi^{(n,k)}, \hat R_\emptyset^{(n,k)})$, such that for any fixed $k \in \mathbb{N}_+$, 
$$\mathbb{E}_n \left[  \left| R_\xi^{(n,k)} -  \hat R_\emptyset^{(n,k)} \right|  \right]   \xrightarrow{P} 0, \qquad n \to \infty.$$
\end{theo}

\begin{proof}
We start by constructing $R_\xi^{(n,k)}$ and $\hat R^{(n,k)}_\emptyset$ according to the coupling described in  Theorem~\ref{T.CouplingTime}, and noting that
\begin{align*}
\mathbb{E}_n\left[ \left| R_\xi^{(n,k)} - \hat R_\emptyset^{(n,k)} \right| \right] &= \mathbb{E}_n\left[ \left| R_\xi^{(n,k)} - \hat R_\emptyset^{(n,k)} \right| 1(\tau \leq k) \right] \\
&\leq \mathbb{E}_n\left[ \left| R_\xi^{(n,k)}  \right| 1(\tau \leq k) \right] + \mathbb{E}_n\left[ \left|  \hat R_\emptyset^{(n,k)} \right| 1(\tau \leq k) \right]. 
\end{align*}
To bound the second expectation, let $M > 0$ be a constant and note that
\begin{align*}
\mathbb{E}_n\left[ \left|  \hat R_\emptyset^{(n,k)} \right| 1(\tau \leq k) \right] &\leq \mathbb{E}_n\left[ \left|  \hat R_\emptyset^{(n,k)} \right| 1(|\hat R_\emptyset^{(n,k)}| > M) \right] + M \mathbb{P}_n \left( \tau \leq k \right).
\end{align*}
Now use Theorems~\ref{T.CouplingTime} and \ref{T.TreeConvergence} to obtain that
$$\limsup_{n \to \infty} \mathbb{E}_n\left[ \left|  \hat R_\emptyset^{(n,k)} \right| 1(\tau \leq k) \right]  \leq E\left[ |\mathcal{R}^{(k)}| 1(|\mathcal{R}^{(k)}| > M) \right],$$
(provided $M$ is a continuity point of the limiting distribution), where $\mathcal{R}^{(k)}$ is defined in Theorem~\ref{T.TreeConvergence} and satisfies $E[|\mathcal{R}^{(k)}|] < \infty$. Taking $M \to \infty$ yields
$$ \mathbb{E}_n\left[ \left|  \hat R_\emptyset^{(n,k)} \right| 1(\tau \leq k) \right]  \xrightarrow{P} 0, \qquad n \to \infty.$$

To show that $\mathbb{E}_n\left[ \left| R_\xi^{(n,k)}  \right| 1(\tau \leq k) \right]$ also converges to zero, we start by using the upper bound $|\zeta_i| \leq c$ for all $1 \leq i \leq n$, to bound $R_\xi^{(n,k)}$ as follows:
$$ \left| R_\xi^{(n,k)}  \right| \leq Y_\xi^{(n,k)} := (|Q_1|, \dots, |Q_n|) \sum_{k=0}^t c^k ({\bf P}^k)_{\bullet \xi},$$
where ${\bf P} = \text{diag}({\bf D}) {\bf A}$, ${\bf D} = (1/(D_1^- \vee 1), \dots, 1/(D_n^- \vee 1))$, and ${\bf P}_{\bullet i}$ is the $i$th column of matrix ${\bf P}$. Define also its coupled version on a tree
$$\hat Y_{\emptyset}^{(n,k)} = \sum_{r=0}^k \sum_{{\bf j} \in \hat A_r} \tilde \Pi_{\bf j} |\hat Q_{\bf j}|,$$
where $\tilde \Pi_\emptyset \equiv 1$ and $\tilde \Pi_{({\bf i},j)}  = \tilde \Pi_{\bf i} \tilde C_{({\bf i},j)}$, with $\tilde C_{\bf i} = c/(\hat D_{\bf i} \vee 1)$.  

Note that the exploration process leading to the coupling of $R_\xi^{(n,k)}$ and $\hat R_\emptyset^{(n,k)}$ also provides a coupling for $Y_\xi^{(n,k)}$ and $\hat Y_\emptyset^{(n,k)}$. In addition, while exploring the in-component of vertex $\xi$, we will also explore the in-components of any vertices we encounter in the process. For vertices that are not in the in-component of $\xi$, we can construct their own couplings with a marked tree using Theorem~\ref{T.CouplingTime}, so that for each vertex in $\mathcal{G}_n$ we obtain a coupling of the form $(Y_i^{(n,k)}, \hat Y_{\emptyset(i)}^{(n,k)})$, $1 \leq i \leq n$, where the subscript $\emptyset(i)$ denotes that the root of the tree where $Y_{\emptyset(i)}^{(n,k)}$ is constructed corresponds to the exploration started at vertex $i$.  Note that the $\{ \hat Y_{\emptyset(i)}^{(n,k)}: 1 \leq i \leq n\}$ will not be independent, since we may have that $\hat Y_{\emptyset(i)}^{(n,r)}$ is constructed on a subtree of a tree rooted at $\emptyset(j)$ for some $r < k$ and $i \neq j$. Using these processes and fixing $M > 0$, we now obtain:
\begin{align*}
&\mathbb{E}_n\left[Y_\xi^{(n,k)} 1(\tau \leq k) \right] \\
&= \mathbb{E}_n \left[ \left( \sum_{j =1}^n c Y_j^{(n,k-1)} {\bf P}_{j\xi}  + |Q_\xi| \right) 1( \tau \leq  k) \right] \\
&=   c \mathbb{E}_n \left[ \left( \sum_{j =1}^n Y_j^{(n,k-1)} 1(Y_j^{(n,k-1)} \neq \hat Y_{\emptyset(j)}^{(n,k-1)}) {\bf P}_{j\xi} + |Q_\xi|   \right) 1( \tau \leq  k) \right]  \\
&\hspace{5mm} +   c \mathbb{E}_n \left[ \left( \sum_{j =1}^n  Y_j^{(n,k-1)} 1(Y_j^{(n,k-1)} = \hat Y_{\emptyset(j)}^{(n,k-1)}){\bf P}_{j\xi}+ |Q_\xi| \right) 1( \tau \leq  k) \right] \\
&= \frac{c}{n} \sum_{i=1}^n \sum_{j=1}^n  \mathbb{E}_{n,i}\left[ Y_j^{(n,k-1)} 1(Y_j^{(n,k-1)} \neq \hat Y_{\emptyset(j)}^{(n,k-1)})  {\bf P}_{ji} 1(\tau \leq k) \right] \\
&\hspace{5mm} +  \frac{c}{n} \sum_{i=1}^n \sum_{j =1}^n\mathbb{E}_{n,i} \left[   \hat Y_j^{(n,k-1)} 1(Y_j^{(n,k-1)} = \hat Y_{\emptyset(j)}^{(n,k-1)}, \hat Y_{\emptyset(j)}^{(n,k-1)} \leq M) {\bf P}_{ji} 1( \tau \leq  k) \right] \\
&\hspace{5mm} +  \frac{c}{n} \sum_{i=1}^n \sum_{j =1}^n \mathbb{E}_{n,i} \left[ \hat Y_j^{(n,k-1)} 1(Y_j^{(n,k-1)} = \hat Y_{\emptyset(j)}^{(n,k-1)}, \hat Y_{\emptyset(j)}^{(n,k-1)} > M) {\bf P}_{ji} 1( \tau \leq  k) \right] \\
&\hspace{5mm} + \frac{c}{n} \sum_{i=1}^n \mathbb{E}_{n,i} \left[  |Q_i|  1( \tau \leq  k) \right]  \\
&\leq \frac{c}{n} \sum_{j=1}^n  \mathbb{E}_{n}\left[ Y_j^{(n,k-1)} 1(Y_j^{(n,k-1)} \neq \hat Y_{\emptyset(j)}^{(n,k-1)})  \right]  + \frac{cM}{n} \sum_{i=1}^n \sum_{j=1}^n \mathbb{E}_{n,i} \left[ {\bf P}_{ji} 1(\tau \leq k) \right]  \\
&\hspace{5mm} + \frac{c}{n} \sum_{j=1}^n \mathbb{E}_n\left[ \hat Y_{\emptyset(j)}^{(n,k)} 1(\hat Y_{\emptyset(j)}^{(n,k)} > M) \right] \\
&\hspace{5mm} + \frac{cM}{n} \sum_{i=1}^n \mathbb{P}_{n,i}(\tau \leq k) + \frac{c}{n} \sum_{i=1}^n  |Q_i| 1(|Q_i| > M) \\
&\leq c \mathbb{E}_n\left[ Y_\xi^{(n,k)} 1(\tau \leq k) \right] + cM  \mathbb{E}_{n} [ D_\xi^- 1(\tau \leq k) ] + c \mathbb{E}_n\left[ \hat Y_\emptyset^{(n,k)} 1(\hat Y_\emptyset^{(n,k)} > M) \right] \\
&\hspace{5mm} + cM \mathbb{P}_n(\tau \leq k) + c \mathbb{E}_n\left[ |\hat Q_\emptyset| 1(|\hat Q_\emptyset| > M) \right],
\end{align*}
where in the first inequality we used the observation that $\sum_{i=1}^n {\bf P}_{ji} \leq 1$, and on the second inequality that $\sum_{j=1}^n {\bf P}_{ji} \leq D_i^-$. It follows that
\begin{align*}
 \mathbb{E}_n\left[ \left| R_\xi^{(n,k)} \right| 1(\tau \leq k) \right] &\leq \mathbb{E}_n\left[Y_\xi^{(n,k)} 1(\tau \leq k) \right] \\
&\leq (1-c)^{-1} \left(  cM \mathbb{E}_n\left[ D_\xi^- 1(\tau \leq k) \right] + c  \mathbb{E}_n\left[ \hat Y_\emptyset^{(n,k)} 1(\hat Y_\emptyset^{(n,k)} > M) \right]   \right. \\
&\hspace{5mm} \left. + cM \mathbb{P}_n(\tau \leq k)  + c \mathbb{E}_n\left[ |\hat Q_\emptyset| 1(|\hat Q_\emptyset| > M) \right] \right).
\end{align*}

To analyze the right-hand side of the inequality, note that Theorem~\ref{T.TreeConvergence} gives (provided $M$ is a continuity point of the limiting distribution):
$$ \mathbb{E}_n\left[ \hat Y_\emptyset^{(n,k)} 1(\hat Y_\emptyset^{(n,k)} > M) \right] \xrightarrow{P} E[ \mathcal{Y}^{(k)} 1(\mathcal{Y}^{(k)} > M) ], \qquad n \to \infty,$$
where $\mathcal{Y}^{(k)}$ is constructed as in Theorem~\ref{T.TreeConvergence} with the generic branching vector $(\mathcal{N}, |\mathcal{Q}|, \mathcal{C}')$, $\mathcal{C}'$ adjusted to match $\tilde C_{\bf i} = c/(D^+_{\bf i} \vee 1)$. Assumption~\ref{A.Primitives} gives (for any continuity point $M$):
$$\mathbb{E}_n\left[ |\hat Q_\emptyset| 1(|\hat Q_\emptyset| > M) \right]  \xrightarrow{P} E[ |\mathcal{Q}_0| 1(|\mathcal{Q}_0| > M) ], \qquad n \to \infty,$$
and Theorem~\ref{T.CouplingTime} gives $\mathbb{P}_n(\tau \leq k) \xrightarrow{P} 0$ as $n \to \infty$. It remains to show that $\mathbb{E}_n\left[ D_\xi^- 1(\tau \leq k) \right]  \xrightarrow{P} 0$ as $n \to \infty$. To do this, note that for both models we have that $D_\xi^-$ converges to $\mathcal{N}_0$ in $d_1$ (use Assumption~\ref{A.Primitives} for the DCM and Theorem~2.4 in \cite{Lee_Olv_19} for the IRD).  Hence, there exists an optimal coupling $(D_\xi^-, \mathcal{N}_0)$ for which the minimum distance is attained, which we use to obtain the bound
\begin{align*}
\mathbb{E}_n\left[ D_\xi^- 1(\tau \leq k) \right] &\leq \mathbb{E}_n\left[ \left| D_\xi^- - \mathcal{N}_0 \right| \right] + \mathbb{E}_n\left[ \mathcal{N}_0 1(\tau \leq k) \right] .
\end{align*}
Theorem~\ref{T.CouplingTime} gives $1(\tau \leq k) \xrightarrow{P} 0$, so dominated convergence (since $E[\mathcal{N}_0] < \infty$) gives
$$\mathbb{E}_n\left[ D_\xi^- 1(\tau \leq k) \right] \xrightarrow{P} 0, \qquad n \to \infty.$$
We conclude that
$$\limsup_{n \to \infty}  \mathbb{E}_n\left[ \left| R_\xi^{(n,k)} \right| 1(\tau \leq k) \right] \leq (1-c)^{-1} \left( c E[ \mathcal{Y}^{(k)} 1(\mathcal{Y}^{(k)} > M) ] + c E[ |\mathcal{Q}_0| 1(|\mathcal{Q}_0| > M) ] \right),$$
and taking $M \to \infty$ completes the proof. 
\end{proof}

\bigskip

Now that we have a coupling between $R_\xi$ and $R_\xi^{(n,k)}$ (Lemma~\ref{L.FiniteIterations}), another one between $R_\xi^{(n,k)}$ and $\hat R_\emptyset^{(n,k)}$ (Theorem~\ref{T.Coupling}), and a third one between $\hat R_\emptyset^{(n,k)}$ and $\mathcal{R}^{(k)}$ (Theorem~\ref{T.TreeConvergence}), it only remains to show that $\mathcal{R}^{(k)}$ converges to $\mathcal{R}^*$.

\begin{lemma} \label{L.TailLimit}
Let
$$\mathcal{R}^{(k)} := \sum_{r=0}^k \sum_{{\bf j} \in A_r} \Pi_{\bf j} \mathcal{Q}_{\bf j} \qquad \text{and} \qquad \mathcal{R}^* := \sum_{r=0}^\infty \sum_{{\bf j} \in A_r} \Pi_{\bf j} \mathcal{Q}_{\bf j}.$$
Then, under Assumption~\ref{A.Primitives}, we have that for any fixed $k \in \mathbb{N}_+$, 
$$E\left[ \left| \mathcal{R}^{(k)} - \mathcal{R}^* \right| \right] \leq \frac{c^{k+1}}{1-c} \cdot E[|\mathcal{Q}_0|] .$$ 
\end{lemma}

\begin{proof}
First note that by Assumption~\ref{A.Primitives} we have, for the DCM:
$$E[ \mathcal{N |C|} ] = E\left[ \mathscr{D}^- \cdot \frac{|\zeta|}{\mathscr{D}^+ \vee 1} \cdot \frac{\mathscr{D}^+}{E[\mathscr{D}^+]} \right] = \frac{1}{E[\mathscr{D}^+]} E[ \mathscr{D}^- |\zeta| 1(\mathscr{D}^+ \geq 1) ] \leq c  < 1,$$
and
$$E[ \mathcal{|CQ|} ] = E\left[ |Q| \cdot \frac{|\zeta|}{\mathscr{D}^+ \vee 1} \cdot \frac{\mathscr{D}^+}{E[\mathscr{D}^+]} \right] = \frac{1}{E[\mathscr{D}^+]} E[ |Q| |\zeta| 1(\mathscr{D}^+ \geq 1) ] \leq \frac{cE[|Q|]}{E[\mathscr{D}^+]} = \frac{cE[|\mathcal{Q}_0|]}{E[\mathcal{N}_0]} < \infty,$$
and for the IRD:
\begin{align*}
E[ \mathcal{N |C|} ] &= E\left[ Z^- \cdot \frac{|\zeta|}{Z^++1} \cdot \frac{W^+}{E[W^+]} \right] = \frac{1}{E[W^+]} E\left[ E[Z^- | W^-] E\left[ \left. \frac{1}{Z^++1} \right| W^+ \right] |\zeta| W^+ \right] \\
&=  \frac{1}{E[W^+]} E\left[ \frac{E[W^+] W^-}{\theta} \cdot \frac{\theta}{E[W^-] W^+} \left( 1- e^{-E[W^-]W^+/\theta} \right) \cdot |\zeta| W^+ \right] \\
&= \frac{1}{E[W^-]} E\left[ |\zeta| W^- \left( 1- e^{-E[W^-]W^+/\theta} \right) \right]  \leq c < 1,
\end{align*}
and
\begin{align*}
E[ \mathcal{|CQ|} ] &= E\left[ Q \cdot \frac{|\zeta|}{Z^++1} \cdot \frac{W^+}{E[W^+]} \right] = \frac{1}{E[W^+]} E\left[  E\left[ \left. \frac{1}{Z^++1} \right| W^+ \right] |Q\zeta| W^+ \right] \\
&=  \frac{1}{E[W^+]} E\left[ \frac{\theta}{E[W^-] W^+} \left( 1- e^{-E[W^-]W^+/\theta} \right) \cdot |Q\zeta| W^+ \right] \leq \frac{\theta cE[|Q|]}{E[W^+] E[W^-]} = \frac{cE[|\mathcal{Q}_0|]}{E[\mathcal{N}_0]} < \infty.
\end{align*}

Now use the branching property to compute:
\begin{align*}
E\left[ \left| \mathcal{R}^{(k)} - \mathcal{R}^* \right| \right] &\leq \sum_{r=k+1}^\infty E\left[ \sum_{{\bf j} \in A_r} |\Pi_{\bf j} \mathcal{Q}_{\bf j}| \right] = \sum_{r=k+1}^\infty E\left[ \sum_{{\bf j} \in A_{r-1}} |\Pi_{\bf j}| \sum_{i=1}^{\mathcal{N}_{\bf j}} |\mathcal{C}_{({\bf j},i)} \mathcal{Q}_{({\bf j},i)} | \right]  \\
&= \sum_{r=k+1}^\infty E\left[ \sum_{{\bf j} \in A_{r-1}} |\Pi_{\bf j}| \mathcal{N}_{\bf j}  \right] E[|\mathcal{CQ}|] = \sum_{r=k+1}^\infty E\left[ \sum_{{\bf j} \in A_0} |\Pi_{\bf j}| \mathcal{N}_{\bf j} \right] \left( E[|\mathcal{C}| \mathcal{N}] \right)^{r-1} E[|\mathcal{CQ} |] \\
&= E[\mathcal{N}_0] E[|\mathcal{CQ}|] \sum_{r=k}^\infty \left( E[|\mathcal{C}| \mathcal{N}] \right)^{r}  \leq cE[|\mathcal{Q}_0|] \sum_{r=k}^\infty c^r = E[|\mathcal{Q}_0|] \frac{c^{k+1}}{1-c}.
\end{align*}
\end{proof}

\bigskip

The proof of Theorem~\ref{T.PageRank} will now immediately follow from combining Lemma~\ref{L.FiniteIterations}, Theorem~\ref{T.TreeConvergence}, Theorem~\ref{T.Coupling},  and Lemma~\ref{L.TailLimit}.

\begin{proof}[Proof of Theorem~\ref{T.PageRank}]
Fix $\epsilon > 0$ and find $k \geq 1$ such that $E[|\mathcal{Q}_0|] (1-c)^{-1} c^{k+1} < \epsilon/2$. Let $H_{n,k}(x) = \mathbb{P}_n( R_\xi^{(n,k)} \leq x)$, $\hat H_{n,k}(x) = \mathbb{P}_n( \hat R_\emptyset^{(n,k)} \leq x)$ and $H_k(x) = P( \mathcal{R}^{(k)} \leq x)$. By Lemma~\ref{L.FiniteIterations} we have
$$d_1(H_n, H_{n,k}) \leq \frac{c^{k+1}}{1-c} \cdot \mathbb{E}_n\left[ |\hat Q_\emptyset| \right].$$
By Theorem~\ref{T.Coupling} we have that
$$d_1(H_{n,k}, \hat H_{n,k}) \xrightarrow{P} 0, \qquad n \to \infty.$$
By Theorem~\ref{T.TreeConvergence} we have that
$$d_1(\hat H_{n,k}, \hat H_{k}) \xrightarrow{P} 0, \qquad n \to \infty.$$
And by Lemma~\ref{L.TailLimit} we have
$$d_1(\hat H_k, H) \leq \frac{c^{k+1}}{1-c} \cdot E[|\mathcal{Q}_0|].$$
It follows from the triangle inequality that
\begin{align*}
d_1(H_n, H) &\leq d_1(H_n, H_{n,k}) + d_1(H_{n,k}, \hat H_{n,k}) + d_1(\hat H_{n,k}, \hat H_k) + d_1(\hat H_k, H)  \\
&\leq \frac{c^{k+1}}{1-c} \cdot \mathbb{E}_n\left[ |\hat Q_\emptyset| \right] +  d_1(H_{n,k}, \hat H_{n,k}) + d_1(\hat H_{n,k}, \hat H_k) + \frac{c^{k+1}}{1-c} \cdot E[|\mathcal{Q}_0|] \\
&\xrightarrow{P} \frac{2c^{k+1}}{1-c} \cdot E[|\mathcal{Q}_0|] < \epsilon
\end{align*}
as $n \to \infty$. Since $\epsilon$ was arbitrary, we conclude $d_1(H_n, H) \xrightarrow{P} 0$ as $n \to \infty$. 
\end{proof}

\subsection{Proofs for results on the power law behavior of PageRank} \label{SS.PowerLaw}

This section contains two results, Theorems~\ref{T.SumFatM} and \ref{T.SumFatY}, which were used to prove Theorems~\ref{T.PageRankRoot}  and \ref{T.PageRankNeighbors}, as well as a proof of the claims made in Example~\ref{E.CorrelatedRV}.  Before we state the theorems, we start with two preliminary technical lemmas related to regularly varying and intermediate regularly varying distributions. Throughout this section, we use $f(x) = O(g(x))$ as $x \to \infty$ to mean $\limsup_{x \to \infty} |f(x)/g(x)| < \infty$.

\begin{lemma} \label{L.IRprop}
Let $f(x) = P(Y > x) \in IR$ and suppose $-\infty < \beta(f) \leq \alpha(f) < 0$ are its Matuszewska indexes. Then, for any $\gamma > (-\beta(f)) \vee 1$ there exists a constant $x_0 > 0$ such that
$$\lim_{x \to \infty} \frac{x^{-\gamma} E[Y 1(x_0 \leq Y < x^\gamma)]}{P(Y > x)}  = 0, \qquad  \text{and} \qquad \lim_{x \to \infty} \frac{P(Y > x^\gamma)}{P(Y > x)} = 0.$$
\end{lemma}

\begin{proof}
For the first limit use Proposition~2.2.1 in \cite{BiGoTe_1987} to obtain that for any $\beta < \beta(f) \leq 0$ (if $\beta(f) > -1$, choose $\beta > -1$), there exists positive constants $H, x_0$ such that for $x \geq x_0$,
\begin{align*}
\frac{E[Y 1(x_0 < Y  \leq x)]}{P(Y > x)} &\leq \frac{x_0 P(Y > x_0)}{P(Y > x)} + \frac{\int_{x_0}^{x} P(Y > t) dt}{P(Y > x)}   \leq H x_0 (x/x_0)^{-\beta} +    H  \int_{x_0}^{x} (x/t)^{-\beta}   dt \\
&\leq H x_0^{\beta+1} x^{-\beta} +  H x^{-\beta} \frac{(x^{\beta+1}-x_0^{\beta+1} )}{\beta +1}  . 
\end{align*}
It follows that 
\begin{align*}
\frac{x^{-\gamma} E[Y 1(x_0 \leq Y < x)]}{P(Y > x)} &\leq H x_0^{\beta+1} x^{-\gamma+|\beta|} + H \left( \frac{ x^{-\gamma+1}}{1-|\beta|} 1(0 \geq \beta > -1) +  \frac{x^{-\gamma+|\beta|} x_0^{\beta+1}}{|\beta|-1} 1(\beta < -1) \right)   \\
&= O\left( x^{-\gamma + (|\beta| \vee 1)} \right)
\end{align*}
as $x \to \infty$. Similarly, Proposition~2.2.1 in \cite{BiGoTe_1987}  also gives that for any $0 > \eta > \alpha(f)$ (if $\alpha(f) < -1$ choose $\eta < -1$), there exists a positive constant $H'$ such that for sufficiently large $x$, 
\begin{align*}
\frac{x^{-\gamma} E[Y 1(x \leq Y < x^\gamma)]}{P(Y > x)} &\leq   x^{-\gamma} \left( x +  \int_{x}^{x^\gamma} \frac{P(Y > t)}{P(Y > x)} dt \right) \leq x^{-\gamma+1} + x^{-\gamma} H \int_{x}^{x^\gamma} (t/x)^{\eta} dt \\
&= x^{-\gamma+1} + H x^{-\gamma-\eta} \frac{(x^{\gamma(\eta+1)} - x^{\eta+1})}{\eta+1} \\
&\leq x^{-\gamma+1} + H \left( \frac{x^{-\gamma-\eta} x^{\gamma(\eta+1)}}{1-|\eta|} 1(0 \geq \eta > -1) +  \frac{x^{-\gamma} x}{|\eta|-1} 1(\eta < -1)  \right) \\
&= O\left( x^{ - (\gamma-1) ( |\eta| \wedge 1)}   \right)
\end{align*}
as $x \to \infty$. Adding the two expressions gives
$$\frac{x^{-\gamma} E[Y 1(x_0 \leq Y < x^\gamma)]}{P(Y > x)} = O\left( x^{-\gamma+(|\beta| \vee 1)} + x^{-(\gamma-1)(|\eta| \wedge 1)} \right) = o(1)$$
as $x \to \infty$.  

The second limit follows from the same inequality used above to obtain for any $0 > \eta > \alpha(f)$ that
$$\frac{P(Y > x^\gamma)}{P(Y > x)} \leq H' \left( \frac{x^\gamma}{x} \right)^{\eta} = H' x^{-(\gamma-1) |\eta|}.$$
\end{proof}

\begin{lemma} \label{L.SumInfM}
Let $\{Z_i\}$ be a sequence of i.i.d.~random variables having the same distribution as $Z$, where $P(Z > x) \sim RV(-\alpha)$ for some $\alpha > 1$, $E[Z] = 0$,  and $E[|Z|^\beta] < \infty$ for all $0 < \beta < \alpha$. Let $M \in \mathbb{N}$ be independent of the $\{Z_i\}$. Assume there exists a random variable $Y$ such that $f(x) = P( Y > x) \in IR$, has Matuszewska indexes $\alpha(f), \beta(f)$ satisfying $-(\alpha \wedge 2) < \beta(f) \leq \alpha(f) < 0$, and is such that $P(M > x) = O(P(Y > x))$ as $x \to \infty$. Then, for any $(-\beta(f)) \vee 1 < \gamma < \alpha \wedge 2$,
$$\lim_{x \to \infty} \frac{P\left( \sum_{i=1}^M Z_i > x, M \leq x^\gamma \right)}{P(Y  > x)} = 0.$$
\end{lemma}

\begin{proof}
Fix $1 < m_0 < \infty$, and  use Burkholder's inequality  to get that for some $(-\beta(f)) \vee 1 < \gamma < \alpha \wedge 2$ we have
\begin{align*}
P\left( \sum_{i=1}^M Z_i > x, M \leq x^\gamma \right) &\leq P\left( \sum_{i=1}^{M} Z_i^+ > x, M \leq m_0 \right) + P\left( \sum_{i=1}^M Z_i > x, m_0 < M \leq x^\gamma \right) \\
&\hspace{5mm} + P(M > x^\gamma)   \\
&\leq P\left( \sum_{i=1}^{m_0} Z_i^+ > x\right) +  E\left[ 1(m_0 < M \leq x^\gamma) P\left( \left. \sum_{i=1}^M Z_i > x \right| M \right) \right]  \\
&\hspace{5mm} + P(M > x^\gamma) \\
&\leq P\left( \sum_{i=1}^{m_0} Z_i^+ > x\right)  + K_\gamma  E\left[ 1(m_0 < M \leq x^\gamma) \frac{M E[|Z]|^\gamma]}{x^\gamma} \right] + P(M > x^\gamma) \\
&= m_0 P(Z > x) (1+o(1)) + O\left(  x^{-\gamma} E[M 1(m_0 < M \leq x^\gamma)] \right) + P(M > x^\gamma)
\end{align*}
for some constant $K_\gamma < \infty$, where in the last step we used the standard heavy-tailed asymptotic for sums of regularly varying i.i.d.~random variables to obtain $P\left( \sum_{i=1}^{m_0} Z_i^+ > x\right) \sim m_0 P(Z > x)$ as $x \to \infty$. To see that $P(Z > x) = o(P(Y > x))$ note that $x^{-\eta} P(Y > x) \to \infty$ for any $\eta < \beta(f)$, so if we choose $-\alpha < \beta < \beta(f)$ and use the fact that $P(Z > x) \in RV(-\alpha)$, we obtain that
$$\lim_{x \to \infty} \frac{P(Z> x)}{P(Y > x)} = \lim_{x \to \infty}  \frac{x^{|\beta|} P(Z> x)}{x^{-\beta} P(Y > x)} = 0.$$
Now choose $m_0 \geq x_0$ according to Lemma~\ref{L.IRprop} to obtain that 
$$x^{-\gamma} E[M 1(m_0< M \leq x^\gamma)] = O\left( x^{-\gamma} E[Y^+ 1(m_0 < Y \leq x^\gamma) ] \right) = o(P(Y > x))$$
and $P(M > x^\gamma) = o(P(M > x))$ as $x \to \infty$ to conclude that
$$\limsup_{x \to \infty} \frac{P\left( \sum_{i=1}^M Z_i > x, M \leq x^\gamma \right)}{P(Y  > x)} = 0.$$
\end{proof}

We are now ready to state and prove Theorem~\ref{T.SumFatM}, which gives the asymptotic behavior of a random sum with regularly varying summands plus a negligible additive term. 

\begin{theo} \label{T.SumFatM}
Let $\{X_i\}$ be a sequence of i.i.d.~random variables having the same distribution as $X$, where $P(X > x) \in RV(-\alpha)$ for some $\alpha > 1$ and $E[|X|^\beta] < \infty$ for all $0 < \beta < \alpha$. Let $(M, Y)$ be independent of the $\{X_i\}$, with $M \in \mathbb{N}$. Assume $P(M > x) \in IR$, $E[X] > 0$, and $P(Y > x) =o(P(M > x))$ as $x \to \infty$; if $E[M] = \infty$ assume further that $f(x) = P(M > x)$ has Matuszewska indexes $\alpha(f), \beta(f)$ satisfying 
 $-(\alpha \wedge 2)  < \beta(f) \leq \alpha(f)  < 0$. Then,
$$P\left( \sum_{i=1}^M X_i + Y > x \right)  \sim 1(E[M] < \infty) E[M] P(X > x) + P(M > x/E[X]), \qquad x \to \infty.$$
\end{theo}

\begin{proof}
Suppose first $E[M] < \infty$ and let $S = \sum_{i=1}^M X_i$. To start, use Theorem~2.5 in \cite{Olvera_12a} to obtain that
\begin{equation} \label{eq:IR}
P\left( S > x \right) \sim E[ M] P(X > x) + P( M > x/E[X] )
\end{equation}
as $x \to \infty$. Now note that \eqref{eq:IR} establishes that $S$ has an IR distribution with a tail at least as heavy as $P(M > x)$. Next, note that for any $0 < \delta < 1$, 
\begin{align*}
P\left( \sum_{i=1}^M X_i + Y > x \right) &\leq P\left( S > (1-\delta) x \right) + P( Y > \delta x ),
\end{align*}
which combined with the assumption $P(Y > x) = o(P(M > x))$ and the properties of the IR class gives
\begin{align*}
&\lim_{\delta \downarrow 0} \limsup_{x \to \infty} \frac{P(S+Y > x)}{E[ M] P(X > x) + P(M > x/E[X] )} \\
&\leq \lim_{\delta \downarrow 0} \limsup_{x \to \infty} \frac{E[M] P(X > (1-\delta)x) + P(M > (1-\delta)x/E[X])}{E[ M] P(X > x) + P(M > x/E[X] )}  = 1.
\end{align*}
Similarly, 
\begin{align*}
P(S+Y > x) &\geq P\left( S + Y >  x, Y \geq -\delta x \right) \geq P\left( S >  (1+\delta) x \right) - P\left( Y < -\delta x \right),
\end{align*}
and we obtain
\begin{align*}
&\lim_{\delta \downarrow 0} \liminf_{x \to \infty} \frac{P(S+Y > x)}{E[ M] P(X > x) + P( M > x/E[X] )} \\
&\geq \lim_{\delta \downarrow 0} \liminf_{x \to \infty} \frac{E[M] P(X > (1+\delta) x) + P(M > (1+\delta) x/E[X])}{E[ M] P(X > x) + P( M > x/E[X] )} = 1.
\end{align*}
This completes the proof for the case $E[M] < \infty$. 

Suppose now that $E[M] = \infty$ and $-\alpha < \beta(f) < 0$. Note that it suffices to show that $P(S > x) \sim P(M > x/E[X])$ as $x \to \infty$.  To obtain an upper bound set $z = (1-\delta)x/E[X]$ and use Lemma~\ref{L.SumInfM} with $Y = M$ 
to obtain
\begin{align*}
P(S > x) &\leq P(S > x, M \leq z) + P(M > z) \\
&\leq P\left( \sum_{i=1}^M (X_i - E[X_i]) > \delta x, M \leq z \right) + P(M > z) \\
&= o\left( P(M > \delta x) \right) + P(M > (1-\delta)x/E[X]) 
\end{align*}
as $x \to \infty$ to conclude that
$$\lim_{\delta \downarrow 0} \limsup_{x \to \infty} \frac{P(S > x)}{P(M > x/E[X])} \leq \lim_{\delta \downarrow 0} \limsup_{x \to \infty}  \frac{P(M > (1-\delta) x/E[X])}{P(M > x/E[X])} = 1.$$

Similarly, we can obtain a lower bound by setting $\hat z = (1+\delta)x/E[X]$ and using Lemma~\ref{L.SumInfM} with $Y = M$ again to obtain
\begin{align*}
P(S > x) &\geq P(S > x, M > \hat z) \\
&\geq P(M > \hat z) - P(M > \hat z, S \leq x)  \\
&= P(M > \hat z) - P\left( \sum_{i=1}^M (ME[X] - X_i)  \geq ME[X] - x, M > \hat z \right) \\
&\geq P(M > \hat z) - P\left( \sum_{i=1}^M (ME[X] - X_i)  \geq \delta x, \, \hat z < M \leq x^\gamma  \right) - P(M > x^\gamma) \\
&= P(M > (1+\delta) x/E[X])  + o(P(M > \delta x)) .
\end{align*}
It follows that
$$\lim_{\delta \downarrow 0} \liminf_{x \to \infty} \frac{P(S > x)}{P(M > x/E[X])} \geq \lim_{\delta \downarrow 0} \liminf_{x \to \infty}  \frac{P(M > (1+\delta) x/E[X])}{P(M > x/E[X])} = 1.$$
This completes the proof for the case $E[M] = \infty$. 
\end{proof}

\bigskip

We now state and prove Theorem~\ref{T.SumFatY} which establishes a similar result for a random sum plus an additive term, except this time the additive term is not negligible.

\begin{theo} \label{T.SumFatY}
Let $\{X_i\}$ be a sequence of i.i.d.~random variables having the same distribution as $X$, where $P(X > x) \sim RV(-\alpha)$ for some $\alpha > 1$ and $E[|X|^\beta] < \infty$ for all $0 < \beta < \alpha$. Let $(M, Y)$ be independent of the $\{X_i\}$, with $M \in \mathbb{N}$. Assume $P(Y > x) \in IR$ and $P(M > x) = o(P(Y > x))$ as $x \to \infty$; if $E[M] = \infty$ assume further that $f(x) = P(Y > x)$ has Matuszewska indexes $\alpha(f), \beta(f)$ satisfying $-(\alpha \wedge 2) < \beta(f) \leq \alpha(f) < 0$. Then, 
$$P\left( \sum_{i=1}^M X_i + Y > x \right)  \sim 1(E[M] < \infty) E[M] P(X > x) + P(Y > x), \qquad x \to \infty.$$  
\end{theo}

\begin{proof}
We start by deriving an upper bound for $P(S+Y > x)$. To this end, fix $0 < \delta < 1$, set $z = 1(E[X] > 0) (1-\delta)\delta x/E[X] + x 1(E[X] \leq 0)$ and note that
\begin{align}
P(S + Y > x) &\leq P(S + Y > x, M \leq z) + P(M > z)  \notag \\
&\leq P\left(S + Y > x, S > (1-\delta)x, M \leq z \right) \notag \\
&\hspace{5mm} + P\left(S + Y > x, S \leq (1-\delta)x, Y > (1-\delta)x, M \leq z \right) \notag \\
&\hspace{5mm} + P(S + Y > x, S \leq (1-\delta)x, Y \leq (1-\delta)x, M \leq z ) + P(M > z) \notag \\
&\leq P(S > (1-\delta) x, M \leq z) + P(Y > (1-\delta)x) \notag \\
&\hspace{5mm} + P\left( S+Y > x, S \leq (1-\delta)x, \delta x < Y \leq (1-\delta)x, M \leq z \right) + P(M > z) \notag \\
&\leq P(S > (1-\delta) x, M \leq z) + P( Y > (1-\delta)x) \notag \\
&\hspace{5mm} + P( S > \delta x, Y > \delta x, M \leq z ) + P(M > z). \label{eq:SYupper}
\end{align}

We will start by analyzing \eqref{eq:SYupper} for the case $E[M] < \infty$, for which we can use Proposition~3.1 in \cite{Olvera_12a} (with $\beta = \alpha$) to obtain that, for $\mu := E[X] \geq 0$, any $0 < \delta < 1$ and $x$ sufficiently large, 
\begin{align*}
P( S > \delta x, Y > \delta x ) &\leq E\left[ 1(Y > \delta x) P( S > \delta x | M) \right] \\
&\leq E\left[ 1(Y > \delta x) \left\{ M P(X > (1-\delta)\delta x) + 1( (\mu+\delta) M > \delta x)  \right\}  \right] \\
&\hspace{5mm} + K E\left[ 1(Y > \delta x)  \left\{ (M+1) P(X > x) x^{-\varepsilon} + x^{-\varepsilon'} 1(M > \delta x/(\log (\delta x))) \right. \right. \\
&\hspace{5mm} \left. \left. + e^{-\varepsilon'' \sqrt{\log x}} 1(M > \delta x/(2\mu)) 1(\mu > 0)  \right\} \right] \\
&\leq E[M 1(Y > \delta x)] P(X > (1-\delta)\delta x) + P((\mu+\delta) M > \delta x) \\
&\hspace{5mm} + K \left( x^{-\varepsilon'} + e^{-\varepsilon'' \sqrt{\log x}} \right) P(Y > \delta x)  + KE[M+1] P(X > x) x^{-\varepsilon}  \\
&= o\left( P(X > x) + P( Y > x) \right)
\end{align*}
for some constants $K, \varepsilon, \varepsilon', \varepsilon'' > 0$, where in the last step we used dominated convergence to obtain $E[M 1(Y > \delta x)] \to 0$, and the assumption $P(M > x) = o(P(Y > x))$.  For the case $E[X] < 0$ use Proposition~3.1 in \cite{Olvera_12a} to obtain that 
\begin{align*}
P( S > \delta x, Y > \delta x ) &\leq E\left[ 1(Y > \delta x) P( S > \delta x | M) \right] \\
&\leq E\left[ 1(Y > \delta x) M P(X > (1-\delta) \delta x) (1+o(1))  \right] = o(P(X > x)).
\end{align*}
It follows that
\begin{equation} \label{eq:Intersection}
\limsup_{x \to \infty} \frac{P(S > \delta x, Y > \delta x)}{P(X > x) + P(Y > x)} = 0. 
\end{equation}
To obtain an upper bound for $P(S > (1-\delta)x, M \leq z)$, use Proposition~3.1 in \cite{Olvera_12a} again to obtain that for $E[X] \geq 0$,
\begin{align*}
P(S > (1-\delta)x, M \leq z) &\leq E\left[ M P(X > (1-\delta)^2 x) \right] + P( (\mu+\delta) M > (1-\delta) x) \\
&\hspace{5mm} + K \left( E[M+1] x^{-\varepsilon} P(X > x) + x^{-\varepsilon'} P\left(M > (1-\delta)x/\log((1-\delta)x) \right)  \right. \\
&\hspace{5mm} \left. + e^{-\varepsilon'' \sqrt{\log x}} P(M > (1-\delta) x/(2\mu)) 1(\mu > 0)  \right) \\
&= E\left[ M  \right] P(X > (1-\delta)^2 x) + o\left( P(X > x) + P(Y > x)  \right) \\
&\hspace{5mm} + o\left( x^{-\varepsilon'} P(Y > (1-\delta)x/\log x) \right) .
\end{align*}
To see that $x^{-\varepsilon'} P(Y > (1-\delta) x/\log x) = o( P(Y > x) )$ as $x \to \infty$, note that Proposition~2.2.1 in \cite{BiGoTe_1987}  gives that for any $\beta < \beta(f) < 0$ there exists a constant $H > 0$ such that
$$\frac{x^{-\varepsilon'} P(Y > (1-\delta) x/\log x)}{P(Y > x)} \leq x^{-\varepsilon'} H \left( \frac{\log x}{1-\delta} \right)^{-\beta} \to 0, \qquad x \to \infty.$$
For $E[X] < 0$, Proposition~3.1 in \cite{Olvera_12a} gives
$$P(S > (1-\delta)x, M \leq z) \leq E\left[ M P(X > (1-\delta)^2 x) \right] (1+o(1)) .$$
Now combine these observations with \eqref{eq:SYupper} to obtain that, for $E[M] < \infty$, 
\begin{align*}
&\lim_{\delta \downarrow 0} \limsup_{x \to \infty} \frac{P(S + Y > x)}{E[M] P(X > x) + P(Y > x)} \\
&\leq \lim_{\delta \downarrow 0} \limsup_{x \to \infty} \frac{E[M] P(X > (1-\delta)^2x) + P(Y > (1-\delta)x)}{E[M]P(X > x) + P(Y > x)} = 1. 
\end{align*}

Now suppose that $E[M] = \infty$ and $-(\alpha \wedge 2) < \beta(f) < 0$. Note that Lemma~\ref{L.SumInfM} gives
\begin{align*}
P(S +Y > x) &\leq P(S+Y > x, S \leq \delta x) + P(S > \delta x, M \leq z) + P(M > z) \\
&\leq P(Y > (1-\delta)x) + P\left( \sum_{i=1}^M (X_i - E[X])  > \delta^2 x, M \leq z \right) + P(M > z) \\
&= P(Y > (1-\delta)x) + o(P(Y > \delta^2 x))
\end{align*}
as $x \to \infty$. Hence,
\begin{align*}
&\lim_{\delta \downarrow 0} \limsup_{x \to \infty} \frac{P(S + Y > x)}{P(Y > x)} \leq \lim_{\delta \downarrow 0} \limsup_{x \to \infty} \frac{ P(Y > (1-\delta)x)}{P(Y > x)} = 1. 
\end{align*}

It follows that for $E[M] \leq \infty$ we have
$$ \limsup_{x \to \infty} \frac{P(S + Y > x)}{1(E[M] < \infty) E[M] P(X > x) + P(Y > x)} \leq 1.$$

We now proceed to prove a lower bound for $P(S+Y > x)$. Similarly as for the upper bound, we have for any fixed $0 < \delta < 1$ and $\hat z = 1(E[X] < 0) (1-\delta)\delta x/|E[X]| + x1(E[X] \geq 0)$, 
\begin{align}
P(S+Y > x) &\geq P( S + Y > x, S > (1+\delta) x) + P( S + Y > x, S \leq (1+\delta) x, Y > (1+\delta) x) \notag \\
&\geq P(S > (1+\delta) x) - P(S+Y \leq x, S > (1+\delta) x) + P(Y > (1+\delta) x) \notag \\
&\hspace{5mm} - P\left(  \{ S+Y > x, S \leq (1+\delta)x \}^c \cap \{ Y > (1+\delta)x \} \right) \notag \\
&\geq P(S > (1+\delta) x)  + P(Y > (1+\delta) x) - P(S > x, Y \leq -\delta x) \notag  \\
&\hspace{5mm} - P(S+Y \leq x, Y > (1+\delta)x) - P(S > (1+\delta) x, Y > (1+\delta)x) \notag  \\
&\geq P(S > (1+\delta) x)  + P(Y > (1+\delta) x) - P(S > x, Y \leq -\delta x) \notag \\
&\hspace{5mm} - P(S \leq -\delta x, Y > x, M \leq \hat z) - P(M > \hat z) - P(S > x, Y > x) . \notag
\end{align}

We start again by assuming $E[M] < \infty$ and noting that by \eqref{eq:Intersection} we have $P(S > x, Y > x) = o(P(X > x) + P(Y > x))$ as $x \to \infty$, and by assumption, $P(M > \hat z) = o(P(Y > x))$. Moreover, the same arguments leading to \eqref{eq:Intersection} also yield, for $E[X] > 0$, 
\begin{align*}
P( S > x, Y \leq -\delta x) &\leq P(X > (1-\delta)x) E[M 1(Y \leq -\delta x)] + P( (\mu+\delta) M > x) \\
&\hspace{5mm} + K\left( x^{-\varepsilon} P(X > x) E[M+1] +  x^{-\varepsilon'} P(M > \delta x/\log x) \right. \\
&\hspace{5mm} \left. + e^{-\varepsilon'' \sqrt{\log x}} P( M > \delta x/(2\mu)) 1(\mu > 0) \right) \\
&= o\left( P(X > x) + P(Y > x) + x^{-\varepsilon'} P(Y > \delta x/\log x)  \right) \\
&= o\left( P(X > x) + P(Y > x)  \right),
\end{align*}
and for $E[X] < 0$, 
\begin{align*}
P( S > x, Y \leq -\delta x)  &\leq E\left[ 1(Y \leq -\delta x) M P(X > (1-\delta) x) (1+o(1)) \right] = o( P(X > x)).
\end{align*}
Therefore, $P(S > x, Y \leq -\delta x) = o(P(X > x) + P(Y > x))$ as $x \to \infty$. To show that $P(S \leq -\delta x, Y > x, M \leq \hat z) = o(P(X > x) + P(Y > x))$, note that
\begin{align*}
P(S \leq -\delta x, Y > x, M \leq \hat z) &= P\left(\sum_{i=1}^M (E[X] - X_i ) \geq ME[X] + \delta x, Y > x, M \leq \hat z \right) \\
&\leq P\left(\sum_{i=1}^M (E[X] - X_i ) \geq \delta^2 x, Y > x, M \leq \hat z \right).
\end{align*}
Now use Burkholder's inequality with $1 < \gamma < \alpha \wedge 2$ to obtain that for some constant $K_\gamma < \infty$, 
\begin{align*}
 P\left(\sum_{i=1}^M (E[X] - X_i ) \geq \delta^2 x, Y > x, M \leq \hat z \right) &\leq K_\gamma E\left[ 1(Y > x, M \leq \hat z) \frac{M E[|X - E[X]|^\gamma]}{(\delta^2 x)^\gamma} \right]  \\
 &\leq \frac{K_\gamma E[|X-E[X]|^\gamma]  \hat z}{\delta^{2\gamma} x^\gamma} P(Y > x) = o(P(Y > x))
\end{align*}
as $x \to \infty$. It follows that $P(S \leq -\delta x, Y > x, M \leq \hat z)  = o(P(Y > x))$ as $x \to \infty$. To complete the analysis of the lower bound for the case $E[M] < \infty$, use Lemma~4.2 in \cite{Olvera_12a} to obtain that
\begin{align*}
P(S > (1+\delta)x) &\geq E\left[ 1(M \leq x/\log x) M P(X > (1+\delta)^2x) \right] + o(P(X > x)).
\end{align*}
We conclude that for $E[M] < \infty$,
\begin{align*}
&\lim_{\delta \downarrow 0} \liminf_{x \to \infty} \frac{P(S + Y > x)}{E[M] P(X > x) + P(Y > x)} \\
&\geq \lim_{\delta \downarrow 0} \liminf_{x \to \infty} \frac{E[M 1(M \leq x/\log x)] P(X > (1+\delta)^2 x) + P(Y > (1+\delta)x)}{E[M] P(X > x) + P(Y > x)} = 1.
\end{align*}

Finally, suppose that $E[M ] = \infty$ and $-(\alpha \wedge 2) < \beta(f) < 0$ and note that
\begin{align*}
P(S + Y > x) &\geq P(S+Y > x, Y > (1+\delta) x) \\
&\geq P(Y > (1+\delta)x) - P(Y > (1+\delta)x, S+Y \leq x ) \\
&\geq  P(Y > (1+\delta) x) - P(S \leq -\delta x, M \leq \hat z) - P(M > \hat z) \\
&\geq P(Y > (1+\delta) x) - P\left( \sum_{i=1}^M (E[X] - X_i )  \geq \delta^2 x , M \leq \hat z \right) - P(M > \hat z).
\end{align*}
Now use Lemma~\ref{L.SumInfM} and the observation that $P(M > \hat z) = o(P(Y > x))$ to obtain that 
$$\lim_{\delta \downarrow 0} \liminf_{x \to \infty} \frac{P(S+Y > x)}{P(Y > x)} \geq \lim_{\delta \downarrow 0} \liminf_{x \to \infty} \frac{P(Y > (1+\delta) x)}{P(Y > x)} = 1.$$
We conclude that for $E[M] \leq \infty$ we have
$$\liminf_{x \to \infty} \frac{P(S+Y > x)}{1(E[M] < \infty) E[M] P(X > x) + P(Y > x)} \geq 1.$$
This completes the proof of the theorem. 
\end{proof}

\bigskip

The last proof in the paper corresponds to the claims made in Example~\ref{E.CorrelatedRV}, which stated that $P(\mathcal{CN} > x)$ and $P(\mathcal{CQ}> 0)$ become negligible with respect to either $P(\mathcal{N}_0 > x)$ or $P(\mathcal{Q}_0 > x)$, respectively, in the presence of in-degree/out-degree dependence.

\begin{proof}[Proof of the claims in Example~\ref{E.CorrelatedRV}]
To prove the first claim note that
 $$P(\mathcal{CN} > x) = \begin{cases} E\left[ 1\left(  \frac{\zeta \mathscr{D}^- }{\mathscr{D}^+ \vee 1}  > x \right) \cdot \frac{\mathscr{D}^+}{E[\mathscr{D}^+]} \right], & \text{for the DCM}, \vspace{2mm} \\
 E\left[ 1\left( \frac{\zeta Z^-}{Z^++1} > x \right) \cdot \frac{W^+}{E[W^+]} \right], &  \text{for the IRD} .
 \end{cases}$$
 It follows, for the DCM, that if we let $a(x)^2 = P(\mathscr{D}^- > x)/P(\mathscr{D}^-/(\mathscr{D}^+\vee 1) > x)$, then
 \begin{align*}
 P(\mathcal{CN} > x) &\leq \frac{1}{E[\mathscr{D}^+]} \left( a(x) P( \mathscr{D}^-/(\mathscr{D}^+ \vee 1) > x) + \int_{a(x)}^\infty P\left( \frac{\mathscr{D}^-}{\mathscr{D}^+ \vee 1} > x, \mathscr{D}^+ > t \right) dt  \right) \\
 &\leq \frac{1}{E[\mathscr{D}^+]} \left( a(x)^{-1} P(\mathscr{D}^- > x) + \int_{a(x)}^\infty P\left(\mathscr{D}^-> tx \right) dt  \right) \\
 &= \frac{P(\mathcal{N}_0 > x) }{E[\mathscr{D}^+]} \left( a(x)^{-1} + \int_{a(x)}^\infty t^{-\alpha} (1+o(1))  \, dt  \right) = o(P(\mathcal{N}_0 > x))
 \end{align*}
 as $x \to \infty$. For the IRD, let $\kappa = E[W^-]/E[W^+ + W^-]$, $\overline{\kappa} = 1- \kappa$, and note that by the union bound we have
\begin{align*}
P(\mathcal{CN} > x) &\leq E\left[ 1\left( \frac{Z^- }{Z^+ + 1} > x, Z^- \leq 2\overline{\kappa} W^-  \right) \cdot \frac{W^+}{E[W^+]} \right] \\
&\hspace{5mm} + E\left[ 1\left( Z^- \geq \max\{ x, 2\overline{\kappa} W^-\}  \right) \cdot \frac{W^+}{E[W^+]} \right]  \\
&\leq \frac{1}{E[W^+]} \left\{ E\left[ 1\left( \frac{2\overline{\kappa} W^-}{Z^+ + 1} > x,  \frac{2\overline{\kappa} W^-}{x}  \leq \frac{\kappa W^++1}{2}  \right) \cdot W^+ \right] \right. \\
&\hspace{5mm}  + E\left[ 1\left(  \frac{2\overline{\kappa} W^-}{x} > \frac{\kappa W^++1}{2}  \right) \cdot W^+ \right]  \\
&\hspace{5mm} \left. + E\left[ 1\left( Z^- - \overline{\kappa} W^- \geq \max\{ x - \overline{\kappa}W^-, \overline{\kappa} W^-\}  \right) \cdot W^+ \right] \right\}.
 \end{align*}
 Now use the inequality $P(X -\lambda \geq x) \vee P( X- \lambda \leq -x) \leq e^{-\frac{x^2}{\lambda+x}} $ for a Poisson r.v.~with mean $\lambda$ and $x > 0$ to obtain that
\begin{align*}
&E\left[ 1\left( Z^- - \overline{\kappa} W^- \geq \max\{ x - \overline{\kappa}W^-, \overline{\kappa} W^-\}  \right) \cdot W^+ \right]  \\
&\leq E\left[ e^{-( \max\{ x - \overline{\kappa}W^-, \overline{\kappa} W^-\})^2/(\overline{\kappa} W^- + \max\{ x - \overline{\kappa}W^-, \overline{\kappa} W^-\})}  \cdot W^+ \right]  \\
&= E\left[ e^{-\overline{\kappa} W^-/2} 1(2\overline{\kappa}W^- \geq x) \cdot W^+ \right]  + E\left[ e^{-(x - \overline{\kappa}W^-)^2/x} 1(2\overline{\kappa}W^- < x) \cdot W^+ \right]  \\
&\leq e^{-x/4} E[W^+] = o\left( P(W^- > x) \right) 
\end{align*}
as $x \to \infty$.  Also, the same arguments used for the DCM give
$$E\left[ 1\left(  \frac{2\overline{\kappa} W^-}{x} > \frac{\kappa W^++1}{2}  \right) \cdot W^+ \right] \leq E\left[ 1\left(  \frac{ W^-}{W^++1} > \frac{\kappa x}{4\overline{\kappa}}  \right) \cdot W^+ \right]  = o\left( P(W^- > x) \right)$$
as $x \to \infty$. For the remaining expectation use the same inequality for the Poisson distribution used earlier to obtain that
\begin{align*}
&E\left[ 1\left( \frac{2\overline{\kappa} W^-}{Z^+ + 1} > x,  \frac{2\overline{\kappa} W^-}{x}  \leq \frac{\kappa W^++1}{2}  \right) \cdot W^+ \right] \\
&\leq E\left[ P( Z^+ - \kappa W^+ \leq -(\kappa W^+ + 1 - 2\overline{\kappa} W^-/x) | W^-, W^+)   1\left(1 <  \frac{2\overline{\kappa} W^-}{x}  \leq \frac{\kappa W^++1}{2}  \right) \cdot W^+ \right] \\
&\leq E\left[  e^{- \frac{(\kappa W^+ +1 - 2\overline{\kappa} W^-/x)^2}{2\kappa W^+  +1 - 2\overline{\kappa} W^-/x} } \cdot 1\left( 1 < \frac{2\overline{\kappa} W^-}{x}  \leq \frac{\kappa W^++1}{2}  \right) \cdot W^+   \right] \\
&\leq E\left[ e^{- \frac{(\kappa W^+/2 )^2}{3\kappa W^+/2}  } \cdot 1\left(  2 \overline{\kappa} W^- > x  \right) W^+ \right] \\
&= E\left[ e^{-\kappa W^+/6 } \cdot 1\left(  2 \overline{\kappa} W^- > x  \right) W^+ \right] ,
\end{align*}
where in the third inequality we used that $h(t) = (a-t)^2/(2a-t)$ is decreasing on $[0, a]$ ($a = \kappa W^+$). Next, let $\varphi(x) = \sup_{t \geq x/(2\overline{\kappa} \log x)} P(W^-/(W^+ \vee 1) > t)/P(W^- > t) \to 0$ as $x \to \infty$, fix $q > 1+\alpha$, set $b(x) = \min\{ \log x, \varphi(x)^{-1/q} \}$, and note that
\begin{align*}
&E\left[ e^{-\kappa W^+/6}  1\left( 2\overline{\kappa} W^- > x \right) W^+ \right] \\
&\leq b(x) P(W^- > x/(2\overline{\kappa}), W^+ \leq b(x))  + \sup_{t \geq b(x)} t e^{-\kappa t/6} P( W^- > x/(2\overline{\kappa})) \\
&\leq b(x) P\left( \frac{W^-}{W^+ \vee 1} \cdot b(x) > x/(2\overline{\kappa}) \right) + o\left( P(W^- > x) \right) \\
&\leq  b(x) P( W^- > x/(2\overline{\kappa} b(x)) ) \varphi(x) + o\left( P(W^- > x) \right) \\
&\leq b(x) \varphi(x) \cdot \frac{P(W^- > x/(2\overline{\kappa} b(x)))}{ P(W^- > x)} \cdot P(W^- > x) + o\left( P(W^- > x) \right) .
\end{align*}
Now use Potter's Theorem to obtain that for any $0 < \epsilon < q-1-\alpha$, 
$$ b(x) \varphi(x) \cdot \frac{P(W^- > x/(2\overline{\kappa} b(x)))}{ P(W^- > x)} = o\left( b(x)^{1+\alpha+\epsilon} \varphi(x) \right) = o\left( \varphi(x)^{\frac{q-1-\alpha-\epsilon}{q} } \right) = o(1)$$
as $x \to \infty$, which implies that
$$E\left[ e^{-\kappa W^+/6}  1\left( 2\overline{\kappa} W^- > x \right) W^+ \right]  = o\left( P(W^- > x) \right)$$
as $x \to \infty$. Finally, note that since $Z^- \stackrel{\mathcal{D}}{=} \mathcal{N}_0$ and $P(W^- > x) = O(P(Z^- > x))$ as $x \to \infty$, then
$$P(\mathcal{CN} > x) = o\left( P(W^- > x) \right) = o\left( P(\mathcal{N}_0 > x) \right), \qquad x \to \infty.$$

The proof of the claim for $P(\mathcal{CQ} > x)$ follows the same steps and is therefore omitted. 
 \end{proof}

\bibliographystyle{plain}
\bibliography{DepPageRank}

\end{document}